\documentclass[10pt]{article}
\usepackage{amsmath, amsthm, amssymb, }
\usepackage{float,epsfig}
\usepackage{color}
\usepackage{graphicx}

\begin{document}

\newtheorem{theorem}{\bf Theorem}[section]
\newtheorem{proposition}[theorem]{\bf Proposition}
\newtheorem{definition}[theorem]{\bf Definition}
\newtheorem{corollary}[theorem]{\bf Corollary}
\newtheorem{example}[theorem]{\bf Example}
\newtheorem{remark}[theorem]{\bf Remark}
\newtheorem{lemma}[theorem]{\bf Lemma}
\newcommand{\nrm}[1]{|\!|\!| {#1} |\!|\!|}

\newcommand{\ba}{\begin{array}}
\newcommand{\ea}{\end{array}}
\newcommand{\von}{\vskip 1ex}
\newcommand{\vone}{\vskip 2ex}
\newcommand{\vtwo}{\vskip 4ex}
\newcommand{\dm}[1]{ {\displaystyle{#1} } }

\newcommand{\be}{\begin{equation}}
\newcommand{\ee}{\end{equation}}
\newcommand{\beano}{\begin{eqnarray*}}
\newcommand{\eeano}{\end{eqnarray*}}
\newcommand{\inp}[2]{\langle {#1} ,\,{#2} \rangle}
\def\bmatrix#1{\left[ \begin{matrix} #1 \end{matrix} \right]}
\def \noin{\noindent}
\newcommand{\evenindex}{\Pi_e}



\def \R{{\mathbb R}}
\def \C{{\mathbb C}}
\def \K{{\mathbb K}}
\def \J{{\mathbb J}}
\def \L{\mathcal{L}}

\def \T{{\mathbb T}}
\def \Pb{\mathrm{P}}
\def \N{{\mathbb N}}
\def \Ib{\mathrm{I}}
\def \Ls{{\Lambda}_{m-1}}
\def \Gb{\mathrm{G}}
\def \Hb{\mathrm{H}}
\def \Lam{{\Lambda}}
\def \Qb{\mathrm{Q}}
\def \Rb{\mathrm{R}}
\def \Mb{\mathrm{M}}
\def \norm{\nrm{\cdot}\equiv \nrm{\cdot}}

\def \P{{\mathbb{P}}_m(\C^{n\times n})}
\def \A{{{\mathbb P}_1(\C^{n\times n})}}
\def \H{{\mathbb H}}
\def \L{{\mathbb L}}
\def \G{{\F_{\tt{H}}}}
\def \S{\mathbb{S}}
\def \sigmin{\sigma_{\min}}
\def \elam{\Lambda_{\epsilon}}
\def \slam{\Lambda^{\S}_{\epsilon}}
\def \Ib{\mathrm{I}}
\def \Tb{\mathrm{T}}
\def \d{{\delta}}

\def \Lb{\mathrm{L}}
\def \N{{\mathbb N}}
\def \Ls{{\Lambda}_{m-1}}
\def \Gb{\mathrm{G}}
\def \Hb{\mathrm{H}}
\def \Delta{\triangle}
\def \Rar{\Rightarrow}
\def \p{{\mathsf{p}(\lam; v)}}

\def \D{{\mathbb D}}

\def \tr{\mathrm{Tr}}
\def \cond{\mathrm{cond}}
\def \lam{\lambda}
\def \sig{\sigma}
\def \sign{\mathrm{sign}}

\def \ep{\epsilon}
\def \diag{\mathrm{diag}}
\def \rev{\mathrm{rev}}
\def \vec{\mathrm{vec}}

\def \herm{\mathsf{Herm}}
\def \sym{\mathsf{sym}}
\def \odd{\mathsf{sym}}
\def \en{\mathrm{even}}
\def \rank{\mathrm{rank}}
\def \pf{{\bf Proof: }}
\def \dist{\mathrm{dist}}
\def \rar{\rightarrow}

\def \rank{\mathrm{rank}}
\def \pf{{\bf Proof: }}
\def \dist{\mathrm{dist}}
\def \Re{\mathsf{Re}}
\def \Im{\mathsf{Im}}
\def \re{\mathsf{re}}
\def \im{\mathsf{im}}
\def \sp{\mathsf{Spec}}

\def \sym{\mathsf{sym}}
\def \sksym{\mathsf{skew\mbox{-}sym}}
\def \odd{\mathrm{odd}}
\def \even{\mathrm{even}}
\def \herm{\mathsf{Herm}}
\def \skherm{\mathsf{skew\mbox{-}Herm}}
\def \str{\mathrm{ Struct}}
\def \eproof{$\blacksquare$}
\def \proof{\noin\pf}

\def \bS{{\bf S}}
\def \cA{{\cal A}}
\def \E{{\mathcal E}}
\def \X{{\mathcal X}}
\def \F{{\mathcal F}}
\def \cH{\mathcal{H}}
\def \cJ{\mathcal{J}}
\def \tr{\mathrm{Tr}}
\def \range{\mathrm{Range}}
\def \adj{\star}
\def \diag{\mbox{diag}}
\def \rank{\mbox{rank}}

\def \pal{\mathrm{palindromic}}
\def \palpen{\mathrm{palindromic~~ pencil}}
\def \palpoly{\mathrm{palindromic~~ polynomial}}
\def \odd{\mathrm{odd}}
\def \even{\mathrm{even}}

\title{Preserving spectral properties of structured matrices under structured perturbations}
\author{ Tinku Ganai\thanks{Department of Mathematics,
IIT Kharagpur, India, E-mail:
tinkuganaimath@gmail.com } \, and  \, Bibhas Adhikari\thanks{Corresponding author, Department of Mathematics,
IIT Kharagpur, India, E-mail:
bibhas@maths.iitkgp.ac.in }}
\date{}

\maketitle
\thispagestyle{empty}

\noindent{\bf Abstract.} 
This paper is devoted to the study of preservation of eigenvalues, Jordan structure and complementary invariant subspaces of structured matrices under structured perturbations. Perturbations and structure-preserving perturbations  are determined such that a perturbed matrix reproduces a given subspace as an invariant subspace and preserves a pair of complementary invariant subspaces of the unperturbed matrix. These results are further utilized to obtain structure-preserving perturbations which modify certain eigenvalues of a given structured matrix and reproduce a set of desired eigenvalues while keeping the Jordan chains unchanged. Moreover, a no spillover structured perturbation of a structured matrix is obtained whose rank is equal to the number of eigenvalues (including multiplicities) which are modified, and in addition, preserves the rest of the eigenvalues and the corresponding Jordan chains which need not be known.  The specific structured matrices considered in this paper form Jordan and Lie algebra corresponding to an orthosymmetric scalar product.\\

\noindent{\bf Keywords.} Structured eigenvalue problem, Jordan algebra, Lie algebra, structure preservation, Jordan chain\\


\section{Introduction}\label{sec:1}

A fundamental concern in matrix perturbation theory is to study the change of spectral properties  such as eigenvalues, eigenspaces and invariant subspaces of a matrix with respect to linear perturbation \cite{wilkinson1965algebraic} \cite{lidskii1966perturbation}  \cite{stewart1990matrix} \cite{moro1997lidskii} \cite{alam2003effect} \cite{rump2006eigenvalues} \cite{butta2014structured} \cite{Sun}. Condition number is a measure to quantify and analyze the sensitivity of such problems for structured and unstructured matrices \cite{karow2006structured} \cite{byers2006structured} \cite{trefethen1991pseudospectra} \cite{tisseur2003chart}. One of the interesting results on the preservation of eigenvalues ​​under perturbation is obtained by S. V. Savchenko as follows. An eigenvalue $\lam$ of a matrix $A$ is preserved as an eigenvalue of a perturbed matrix $A+B$ if the geometric multiplicity of $\lam$ is greater than the rank of $B$ \cite{savchenko2004change}. Besides, a necessary and sufficient condition on the entries of a perturbation is given under which the spectral properties of an eigenvalue change. Later, it is proved by Julio Moro and Frolian M. Dopico that except for a set of zero Lebesgue measure, a low rank perturbation $A + B$ of a matrix $A$ destroys for each of its eigenvalues exactly the $\rank (B)$ largest Jordan blocks of $A$, while the rest remain unchanged \cite{moro2003low}. A detailed study of behavior of Jordan structure of variety of structured matrices under rank-one structured perturbations is performed by Mehl-Mehrmann-Ran-Rodman in a series of papers, see \cite{mehl2011eigenvalue} \cite{mehl2012perturbation} \cite{mehl2014eigenvalue} and the references therein. Later in \cite{batzke2016generic} Batzke et al. have studied the effect of structured low rank perturbations on the Jordan structure of certain structured matrices. 

The main objective of this article is to determine structured perturbations of a structured matrix such that the perturbed matrices reproduce a given set of desired eigenvalues while maintaining the invariance of the Jordan structure of the corresponding unperturbed matrix. One of the first studies on the modification of an eigenvalue and preservation of the rest of the eigenvalues of a matrix under rank one perturbation is due to Brauer in 1952 as follows. 
 


\begin{theorem}(Brauer's Theorem (\cite{brauer1952limits},\cite{bru2012brauer}))
Let $A$ be an arbitrary $n\times n$ matrix with eigenvalues $\sig(A) = \{\lam_1, \lam_2, \hdots, \lam_n\}$. Let $x_k$ be an eigenvector of $A$ associated with the eigenvalue $\lam_k$, and let $q$ be any $n$-dimensional vector. Then the matrix $A+x_kq^T$ has eigenvalues $\{\lam_1,\hdots,\lam_{k-1},\mu_k=\lam_k+x_k^Tq, \lam_{k+1},\hdots, \lam_n\}.$
\end{theorem}

It can further be verified that the perturbation given in Brauer's theorem has no effect on the eigenvectors corresponding to the eigenvalue which gets modified due to perturbation, but eigenvectors corresponding to other eigenvalues change \cite{saad1992numerical}. An alternative statement of Brauer's theorem is as follows \cite{chiang2013eigenvalue}.


\begin{theorem}\label{BrauerThm2}
Let $A$ be a matrix and $(\lam,v)$ be an eigenpair of $A.$ If $r$ is a vector so that $r^Tv=1$ then for any scalar $\mu,$ eigenvalues of the matrix $A+(\mu-\lam)vr^T,$ consists of those of $A,$  except that one eigenvalue $\lam$ of $A$ is replaced by $\mu.$ Moreover, the eigenvector $v$ is unchanged, that is, $(A+(\mu-\lam)vr^T)v=\mu v.$ 
\end{theorem}

Moreover, if the given matrix $A$ is symmetric then a structure preserving rank one perturbation can be given for which the conclusion of Theorem \ref{BrauerThm2} remains valid. For instance, setting $\Delta A=(\mu-\lam_k)x_kx_k^T$ where $(\lam_k,x_k)$ is the $k$th eigenpair of the matrix $A$ with $\|x_k\|_2=1,$ the conclusion of Theorem \ref{BrauerThm2} remains valid for the perturbed matrix $A+\Delta A$ \cite{bru2012brauer}. This result is well-known as Hotelling's deflation. Brauer's theorem has found applications to defining many matrices of interest with desired properties using a rank one update. For example, in order to ensure the existence and convergence of the Page Rank power method a stochastic matrix is updated by a rank one matrix to construct the Google matrix, which is a consequence of the Brauer's Theorem \cite{langville2004deeper}. It also plays an important role into the pole assignment of SISO linear time invariant control systems, solving matrix equations in particular quadratic equations \cite{meini2013shift}, QBD and M/G/1-type Markov chains \cite{guo2003comments} \cite{bini2005shift}, model updating problems for structural models \cite{zhang2018explicit}, and solving algebraic Riccati equations \cite{bini2012numerical} to name a few. .

A generalization of Brauer's theorem is known as Rado's theorem, in which several eigenvalues can be modified by a single perturbation into a desired set of eigenvalues ​​of the perturbed matrix while preserving the remaining eigenvalues of the unperturbed matrix. Rado's theorem was presented by Perfect in \cite{perfect1955methods} as follows. 

\begin{theorem}(Rado's Theorem)\label{RadoThm}
Let $A$ be an arbitrary $n\times n$ matrix. Let $\lam_i,  i=1,\hdots, n$ denote the eigenvalues of $A.$ Let $x_j, j=1,\hdots,p$ be a collection of linearly independent eigenvectors of $A$ corresponding to the eigenvalues $\lam_j.$ Then the matrix $\hat{A}=A+XC$ has eigenvalues $\mu_1,\hdots,\mu_p,\lam_{p+1},\hdots,\lam_n,$ where $\mu_j, j=1,\hdots,p$ are eigenvalues of of the matrix $\Omega + CX$ with $\Omega=\mbox{diag}(\lam_1,\hdots,\lam_p),$ $X=[x_1 \,\, x_2 \,\, \hdots \,\, x_p]$ and $C$ is an arbitrary matrix of order $p\times n.$ 
\end{theorem} 

It follows from the statement of Rado's theorem that the eigenvectors corresponding to the eigenvalues $\lam_j,$ $j=1,\hdots,p$ of $A$ need not be eigenvectors corresponding to the eigenvalues $\mu_j$ of $\hat{A},$ since $C$ can be chosen arbitrarily. Instead, under certain conditions, these eigenvectors can be invariant. For example, if $C$ is chosen such that $CX=\mbox{diag}(\mu_1-\lam_1, \mu_2-\lam_2, \hdots, \mu_p-\lam_p)$ then the eigenvectors $x_j$ corresponding to eigenvalues $\lam_j$ of $A$ remain eigenvectors corresponding to the eigenvalues $\mu_j$ of $\hat{A}.$ In this case the perturbation $CX$ has rank $p.$ For a symmetric matrix $A$ of order $n\times n,$ a symmetric perturbation given by $\hat{A}=A+X_p\Theta X_p^T$ preserve the  eigenvectors $x_j$ for the eigenvalues $\mu_j,$ $j=1,\hdots, p$ of $\hat{A}$ while maintaining the invariance of the remaining set of eigenvectors. Here  $\Theta = \mbox{diag}(\mu_1-\lam_1, \mu_2-\lam_2, \hdots, \mu_p-\lam_p)$ and $X_p=[x_1 \,\, x_2 \,\, \hdots \,\, x_p]$ is the matrix of orthonormal eigenvectors corresponding to eigenvalues $\lam_j, j=1,\hdots,p$ of $A.$

An important application of Rado's theorem is to solve inverse eigenvalue problems for nonnegative matrices \cite{soto2006applications} \cite{julio2020role}. Recently, a generalized version of Rado's theorem and hence Brauer's theorem is provided for matrix polynomials in \cite{bini2017generalization}.  An immediate consequence of Rado's theorem is as follows. Let $\mu_j, j=1,\hdots, p$ be a collection of scalars which are given beforehand. Then setting $C=(\diag(\mu_1, \hdots, \mu_p) - \Omega)X^\dagger + Z(I-XX^\dagger)$ it can be checked that $\mu_j$ becomes an eigenvalue of $\hat{A}=A+XC$ corresponding to eigenvector $x_j,$ $j=1,\hdots, p,$ where $Z$ is any arbitrary matrix of dimension $n\times n,$ and $X^\dagger$ denotes the Moore-Penrose pseudoinverse of $X$ while $I$ denotes the identity matrix of compatible size. However the eigenpairs $(\lam_{k}, x_k),$ $k=p+1, \hdots, n$ of $A$ need not be preserved as eigenpairs of the perturbed matrix $\hat{A}.$ Then from the view of perturbation theory the following problem can be formulated. Given $A, \mu_j$ and the  eigenpairs $(\lam_j,x_j), j=1,\hdots,p$ of $A,$ determine perturbations $\Delta A,$ if exists, such that the pairs $(\mu_j, x_j)$ become eigenpairs of the perturbed matrix $\hat{A}=A+\Delta A$ and $\hat{A}$ preserves the remaining eigenpairs of $A$ even if those are not known. Then we call such a perturbation $\Delta A$ as no spillover perturbation, a term which is frequently used in a similar context that arises in the theory of model updating problem for structural models defined by matrix polynomials \cite{adhikari2020updating}. 


Moreover, in practical applications it may so happen that the given matrix $A$ is a structured matrix. For example, in \cite{alam2011perturbation}, Alam et al. have considered the problem of finding structured perturbation of a Hamiltonian matrix such that the perturbed matrix modifies certain eigenvalues of the unperturbed matrix while preserving the rest of the eigenvalues. Then from the view of structured perturbation theory it is desirable to find no spillover perturbations $\Delta A$ such that the perturbed matrices $A+\Delta A$ have the same structure as that of $A.$ To be specific, if $\S$ denotes a set of structured matrices and $A\in \S$ then determine perturbations $\Delta A$ such that $A+\Delta A\in\S.$  Further, a perturbation $\Delta A_o$ is called a minimal perturbation if it has the smallest norm among all perturbations which satisfy the given property. In this paper, we consider Frobenius norm of a matrix. 

In this paper, we consider $\S,$ a set of linearly structured matrices described as follows. Let $H\in \K^{n \times n}$ be a unitary matrix with $H^\star=\epsilon H$ where $\epsilon \in\{1,-1\},\,\star \in \{T,*\}$ and $\K\in\{\R,\C\}.$ Here $*$ denotes the conjugate transpose, and $^T$ denotes transpose. Define the scalar product $\langle \cdot,\cdot \rangle_H$ on $\K^n$ by $\langle x,y \rangle_H:=y^{\star}Hx.$ Then for any given matrix $A\in \K^{n \times n}$ the adjoint of $A,$ denoted by $A^{[\star]}$ with respect to the scalar product $\langle \cdot,\cdot \rangle_H$ satisfies $$\langle Ax,y \rangle_H=\langle x, A^{[\star]}y \rangle_H, \,\,x,y \in \K^n.$$ 
Then it can be seen that 
\begin{align} \label{adjointA}
A^{[\star]}=H^{-1}A^{\star} H=\left\{
	\begin{array}{ll}
		H^{-1}A^T H,  & \mbox{bilinear form, }  \\
		H^{-1}A^* H, & \mbox{sesquilinear form }. 
	\end{array}
\right.
\end{align} 
The Lie Algebra and Jordan Algebra of matrices defined by $\langle \cdot,\cdot \rangle_H$ are given by 
\begin{eqnarray}\label{def:LJ} \L &=& \{A\in \K^{n \times n}\,:\,A^{[\star]}=-A\} \,\, \mbox{and}\\ \J &=& \{A\in \K^{n \times n}\,:\,A^{[\star]}=A\}, \end{eqnarray} respectively. Thus the set of structured matrices which we consider in this paper is $\S\in\{\L, \J\}.$ We mention that the set $\L$ or $\J$ provides a variety of structured matrices which arise in practical applications, for different choices of the matrix $H.$ For instance, if $H=\bmatrix{0 & I_n \\ I_n & 0}$ then $\J$ is the set of Hamiltonian, whereas $\L$ consists of all skew-Hamiltonian matrices. Besides, eigenpairs of $A\in\S$ preserve certain symmetries. Detailed discussions on these structured matrices and its eigenpair symmetries can be found in \cite{karow2006structured} \cite{LieJordan}, and the references therein.  

In this paper we are concerned with finding structured perturbations of a given structured matrix such that a perturbed matrix reproduces a desired set of eigenvalues, and it preserve the Jordan basis and rest of the eigenvalues of the unperturbed matrix. Consequently, such a perturbation preserves the sizes of Jordan blocks of the original matrix, and hence it extends Rado's theorem for structured matrices. However, eigenvalues of a structured matrix occur in pairs due to the symmetry in the structure of the matrix \cite{karow2006structured}. This imposes an additional condition on the scalars which have to be reproduced as eigenvalues of the perturbed matrix. 

Recall that a basis in $\K^n$ is called a Jordan basis of  a matrix $A\in \K^{n\times n}$ if $A$ has a block diagonal form $\diag (J(\lam_1), \hdots, J(\lam_k))$ with respect to this basis and $J(\lam_i)\in\K^{m_i\times m_i}$ is of the form $$\bmatrix{\lam_i & 1 & & & 0 \\ & \lam_i & 1 & & \\ & & \ddots & \ddots & \\ &&&\ddots  & 1 \\ 0 & &&& \lam_i}$$  such that $\sum_{i=1}^k m_i=n,$ where $\lam_i$s are eigenvalues of $A.$ The number of Jordan blocks corresponding to an eigenvalue $\lam_i$ is the geometric multiplicity of $\lam_i.$  A nonzero vector $v\in\K^n$ is called a generalized eigenvector corresponding to an eigenvalue $\lam$ of $A$ if $v\in\ker(A-\lam I)^l$ for the smallest positive integer $l.$ Then the set of linearly independent vectors $\{v, (A-\lam I)v, \hdots, (A-\lam I)^{l-1}v\}$ is called the Jordan chain corresponding to $v.$ A pair $(\lam, V)$ is called a Jordan pair of $A$ if  $V$ is the matrix whose columns are elements of the Jordan chain corresponding to a generalized eigenvector $v$ associated with an eigenvalue $\lam$ of $A.$ Obviously $AV=VJ(\lam)$ where dimension or size of the Jordan block $J(\lam)$ is same as the number of vectors in the associated Jordan chain, denoted by $\# (V)$ or $\# J(\lam).$ 


Thus we define the following problem. 

\noindent{\bf Problem 1 (P1)} \textit{Let $A\in\K^{n\times n},$ and $\lam_i^c,\lam_j^f, i=1,\hdots, p,\, j=p+1,\hdots ,s\leq n$ are distinct eigenvalues of $A.$  Suppose $\{(\lam^c_i ,X^c_{i,l})\}_{l=1}^{m_i}$ is the collection of Jordan pairs of $A$ corresponding to the eigenvalue $\lam_i^c,$  where $m_i$ is the geometric multiplicity of $\lam_i^c.$ Let $\{(\lam_k^f, X_{k,l}^f)\}_{l=1}^{m_k}$, $k=p+1,\hdots, s$ denote the remaining set of Jordan pairs of $A$ where $m_k$ is the geometric multiplicity of $\lam_k^f$ such that $\sum_{i=1}^{p} \sum_{l=1}^{m_i}\#(X^c_{i,l})+\sum_{k=p+1}^{s}\sum_{l=1}^{m_k} \#(X^f_{k,l})=n.$ }


\textit{Then for a given set of scalars $\{\lam_i^a : 1\leq i\leq p\},$ which is closed under desired eigenvalue pairing, determine perturbations $\Delta A$ of $A$ such that $(A+\Delta A)X_{i,l}^c=X_{i,l}^c J_l(\lam_i^a),$ $l=1,\hdots,m_i$ and $A+\Delta A$ preserve the Jordan pairs  $\{(\lam_k^f, X_{k,l}^f)\}_{l=1}^{m_k}$ where $J_l(\lam)$ denotes the Jordan block corresponding to $\lam$ of compatible size. Moreover, if $A\in \S\subset \K^{n\times n}$ then determine  perturbations $\Delta A$ which satisfy the above properties and $A+\Delta A\in \S.$   (The notations $^c, ^f, ^a$ stand for change, fixed, and aimed respectively.) }


\textit{If $(\lam_{k}^f, X_{k,l}^f), k=p+1, \hdots, s$ are not known then determine structured perturbations $\Delta A,$ if exists, which satisfy the condition above. Such a perturbation $\Delta A$ is called a no-spillover perturbation. }

Thus it follows that perturbations which solve {\bf (P1)}, do reproduce a desired set of eigenvalues in the perturbed matrices while preserving the Jordan bases, and hence sizes of the Jordan blocks of the unperturbed matrix. Besides, the perturbed matrices preserve the generalized eigenspaces of the original matrix which are invariant subspaces. Thus a general problem about invariant subspaces can be formulated as follows. Set
\beano \Lam_c &=& \diag(J(\lam_1^c),\hdots, J(\lam_p^c)) \, \mbox{where} \, J(\lam_i^c)=\diag(J_1(\lam_i^c), \, \hdots, J_{m_i}(\lam_i^c)),i=1,\hdots, p \\ 
\Lam_f &=& \diag(J(\lam_{p+1}^f),\hdots, J(\lam_s^f)) \, \mbox{where} \, J(\lam_k^f)=\diag(J_1(\lam_k^f), \, \hdots, J_{m_k}(\lam_k^f)), k=p+1,\hdots, s \\
X_c &=& \bmatrix{X_1^c & \hdots & X_p^c} \, \mbox{where} \, X_i^c=\bmatrix{X_{i,1}^c & \hdots & X_{i,m_i}^c} \\
X_f &=& \bmatrix{X_{p+1}^f & \hdots & X_s^f} \, \mbox{where} \, X_k^c=\bmatrix{X_{k,1}^c & \hdots & X_{k,m_k}^c}. 
 \eeano Then problem {\bf (P1)} can be reformulated as follows. If $A$ satisfies $AX_c=X_c\Lam_c$ and $AX_f=X_f\Lam_f,$ that is $$A\bmatrix{X_c & X_f} = \bmatrix{X_c & X_f}\diag(\Lam_c, \Lam_f)$$ then determine $\Delta A$ such that \begin{equation}\label{eqn:invp}(A+\Delta A)\bmatrix{X_c & X_f}=\bmatrix{X_c & X_f}\diag(\Lam_a, \Lam_f)\end{equation} for a given $\Lambda_{a}=  \diag(J(\lam_1^a),\hdots, J(\lam_p^a)) $ where $J(\lam_i^a)=\diag(J_1(\lam_i^a), \, \hdots, J_{m_i}(\lam_i^a)),$ $\lam_i^a\in\K,$ $i=1,\hdots, p$ and $\#J_l(\lam^a_i)=\# J_l(\lam_i^c),$ $l=1,\hdots, m_i.$

 Recall that a pair of matrices $(X, D)\in \K^{n\times p}\times \K^{p\times p}$ with $\rank(X)=p$ is called an invariant pair of a matrix $A\in\K^{n\times n}$ if $AX=XD.$ Then the range space of $X,$ denoted by $\mathfrak{R}(X),$ is an invariant subspace of $A.$ Thus the problem {\bf (P1)} can be extended into the framework of invariant pairs when the matrices $\Lam_c,\, \Lam_a$ and $\Lam_f$ need not be in block diagonal form with Jordan blocks as its diagonal entries. Further, from equation (\ref{eqn:invp}) it can be assumed that $\bmatrix{X_c & X_f}$ is a nonsingular matrix which generates the entire space $\K^n.$ We call two invariant pairs $(X_c,\Lam_c)$ and $(X_f,\Lam_f)$ of a matrix $A$ are complementary if $\bmatrix{X_c & X_f}$ is nonsingular, and hence $\mathfrak{R}(X_c) \oplus \mathfrak{R}(X_f)=\K^n,$ where $\oplus$ denotes the direct sum of vector spaces.  Then it follows from equation (\ref{eqn:invp}) that the required no spillover perturbations $\Delta A$ must satisfy  \begin{equation}\label{eqn:invp2} \bmatrix{X_c & X_f}^{-1}(A+\Delta A)\bmatrix{X_c & X_f}=\bmatrix{\Lam_a & 0 \\ 0 & \Lam_f}.\end{equation} Thus we define the following problem. \\

\noindent{\bf Problem 2 (P2)} \textit{Let $A\in\K^{n\times n}$ and $\mathcal{X}_c$ be a subspace of $\K^n$ of dimension $p\leq n.$ Then characterize the perturbations $\Delta A\in\K^{n\times n}$ of $A$ such that the perturbed matrices $A+\Delta A$ reproduce $\mathcal{X}_c$ as an invariant subspace. Otherwise, if $\mathcal{X}_c$ is an invariant subspace of $A,$ determine $\Delta A$ such that $\mathcal{X}_c$ remains invariant as an invariant subspace of $A+\Delta A.$ In addition, determine no spillover  perturbations $\Delta A,$ if exists such that $A+\Delta A$ preserve a given complementary pair of invariant subspaces of $A.$ Moreover, determine (minimal) structured perturbations $\Delta A$ which solve the above problem for a structured matrix $A.$}

Here we mention that extensive research has been performed in literature to quantify the effect of perturbation of a matrix on the invariant subspaces of the matrix \cite{Sun} \cite{byers2006structured}. While on the contrary, the problem {\bf (P2)} is concerned with determining (structure preserving) perturbations of a matrix such that a pair of complementary invariant subspaces of the matrix are preserved under (structured) perturbations of the matrix.    

The contribution of this work are as follows. First we consider {\bf (P2)}. Given an unstructured matrix $A\in\K^{n\times n}$ and a subspace $\mathcal{X}$ of $\K^n,$ we determine all perturbations $\Delta A$ such that $\mathcal{X}$ becomes an invariant subspace of $A+\Delta A.$ Utilizing this result parametric representation of all perturbations are obtained which preserve complementary invariant subspaces of a given matrix under perturbation. A necessary and sufficient condition is obtained that guarantees the existence of structure-preserving perturbations which perform the same task for a structured matrix. Consequently, structured perturbation and minimal structured perturbation are determined that preserve a pair of complementary invariant subspaces of a given structured matrix. In addition, if the given subspace is of dimension $p$ then a no spillover structured perturbation of rank less equal to $p$ is obtained for a structured matrix.  Next we consider {\bf (P1)}. Orthogonality properties of Jordan chains corresponding to eigenvalues of a structured matrix are derived. These properties are used to characterize all structured perturbations which modify a given set of eigenvalues of a structured matrix by replacing them with a desired set of scalars while keeping the Jordan structures of the those eigenvalues unchanged. Further, analytical expression of a structured no spillover perturbation of rank $p$ is obtained when the number of eigenvalues (including multiplicities) to be changed under perturbation is $p.$ Thereby it is shown that the sufficient condition for eigenvalue preservation based on rank of the perturbation and geometric multiplicity of the eigenvalue to be preserved as derived by  S. V. Savchenko in \cite{savchenko2004change} is not a necessary condition for eigenvalue preservation under structure-preserving perturbation of a structured matrix. \\

\textbf{Notation.} $\K$ denotes the field of real numbers $\R$ or the field of complex numbers $\C$. $i\R$ denotes the set of purely imaginary numbers. $\sigma(\Lam)$ denotes the spectrum (multi-set of eigenvalues) of $\Lam$ and $\|X\|_F$ denotes the Frobenius norm of $X$. For a matrix $X,$ $\mathfrak{R}(X)$ denotes the range space or column space of $X.$

\section{Preserving invariant subspaces under perturbations}
 
In this section we determine all perturbations of an unstructured matrix such that a given invariant subspace of the matrix is preserved under the perturbations, and we determine no spillover perturbations which inherit a pair of complementary invariant subspaces of the matrix. Next we extend this result to structured matrices, in which we first find a necessary and sufficient condition satisfying which structured perturbations can be obtained that preserve a given invariant subspace. Thus we consider problem {\bf (P2)} in this section.

\subsection{Unstructured case}

First we recall that given a pair $(X, B)\in\K^{n\times p}\times \K^{n\times p}$ there exists a matrix $A\in \K^{n\times n}$ which satisfies $AX=B$ if and only if $BX^{\dagger}X=B.$ If $X$ has full rank then this condition is obviously satisfied, and any such matrix $A$ is given by \begin{equation}\label{eqn:ls}A=BX^\dagger + Z(I-XX^\dagger),\end{equation} where $Z\in\K^{n\times n}$ is an arbitrary matrix and $I$ is the identity matrix of order $n$ \cite{AdhAlam}. Below we determine all perturbations of an unstructured matrix such that the corresponding perturbed matrices reproduce a given subspace as invariant subspace. Besides, we determine no spillover perturbations of a matrix which preserve complementary pair of invariant spaces of the matrix. Recall that $\mathcal{X}$ is a $p$ dimensional invariant subspace of a matrix $A\in\K^{n\times n}$ if and only if $AX=XD$ for some $D\in\K^{p\times p}$ where $\rank(X)=p$ such that $\mathfrak{R}(X)=\mathcal{X}.$

First we have the following proposition.

\begin{proposition}\label{Prop:uns1}
Let $A\in\K^{n\times n}$ and $\mathcal{X}_a$ be a subspace of $\K^n$ of dimension $p.$ Then any perturbation $\Delta A\in\K^{n\times n}$ such that $\mathcal{X}_a$ is an invariant subspace of $A+\Delta A$ is given by $$\triangle A=(X_a \Lambda_a-A X_a)X^{\dagger}_a+Z(I_n-X_aX_a^{\dagger})$$ for some $\Lam_a\in\K^{p\times p},$ where $X_a\in\K^{n\times p}$ is a full rank matrix such that $\mathfrak{R}(X_a)=\mathcal{X}_a, $ and $Z \in \K^{ n\times n}$ is arbitrary.
\end{proposition}
\begin{proof} Let  $X_a\in\K^{n\times p}$ be a full rank matrix such that $\mathfrak{R}(X_a)=\mathcal{X}_a.$ Then $\mathcal{X}_a$ is an invariant subspace of $A+\Delta A$ for some $\Delta A,$ if $(A+\Delta A)X_a=X_a\Lambda_a$ for some $\Lam_a\in\K^{p\times p}.$ This is satisfied if $\Delta A X_a=X_a\Lam_a-AX_a,$ in which $X_a$ and the residual matrix in the right hand side are known. This is a linear system of the form $AX=B,$ where $X,B$ are known and $A$ is unknown. Then the desired result follows from equation (\ref{eqn:ls}). $\hfill{\square}$
\end{proof}
Now we determine perturbations of a matrix
under which an invariant subspace of the matrix is preserved.
\begin{theorem}\label{Thm:uns1}
Let $A\in\K^{n\times n}.$ Suppose $\mathcal{X}_c$ is an invariant subspace of $A$ of dimension $p$ such that $AX_c=X_c\Lam_c$ where $X_c\in\K^{n\times p}$ is full rank such that $\mathfrak{R}(X_c)=\mathcal{X}_c$ and $\Lam_c\in\K^{p\times p}.$ Then any $\Delta A$ for which $\mathcal{X}_c$ remains an invariant subspace of $A+\Delta A$ is given by $$\Delta A=X_c(R\Lam_a-\Lam_cR)(X_cR)^\dagger +Z(I-X_cR(X_cR)^\dagger)$$ for some $\Lam_a\in\K^{p\times p}$ and nonsingular matrix $R\in\K^{p\times p}$, where $Z\in\K^{n\times n}$ is arbitrary.

In addition, let $\mathcal{X}_f$ be an invariant subspace of $A$ such that $\mathcal{X}_c \oplus \mathcal{X}_f=\K^n.$ Then a perturbation $\Delta A$ for which $\mathcal{X}_c$ and $\mathcal{X}_f$ are complementary invariant subspaces of $A+\Delta A$ is given by $$\triangle A=\bmatrix{X_a\hat{\Lambda}_a-AX_a & X_f\hat{\Lam}_f - AX_f}\bmatrix{X_a & X_f}^{-1}$$ for some $\hat{\Lam}_a\in\K^{p\times p}$ and $\hat{\Lam}_f\in\K^{(n-p)\times (n-p)}$, where $X_a$ and $X_f$ are full rank matrices such that $\mathfrak{R}(X_a)=\mathcal{X}_c$ and $\mathfrak{R}(X_f)=\mathcal{X}_f$ respectively. 
\end{theorem}

\begin{proof}
$\mathcal{X}_c$ is an invariant subspace of $A+\Delta A$ for some $\Delta A\in\K^{n\times n}$ if there exists a full rank matrix $X_a\in\K^{n\times p}$ such that $\mathfrak{R}(X_a)=\mathcal{X}_c$ and $(A+\Delta A)X_a=X_a\Lam_a$ for some $\Lam_a\in\K^{p\times p}.$ Further $X_a=X_cR$ for some nonsingular matrix $R\in\K^{p\times p}$ since $\mathfrak{R}(X_c)=\mathcal{X}_c.$ Then $AX_a=AX_cR=X_c\Lam_cR.$ Thus $\Delta A$ has to satisfy $$\Delta AX_a= X_a\Lam_a-AX_a=X_c(R\Lam_a-\Lam_cR).$$ Then the desired result follows using a similar argument as in Proposition \ref{Prop:uns1}. 

Next assume that $X_f\in\K^{n\times (n-p)}$ is a full rank matrix such that $\mathfrak{R}(X_f)=\mathcal{X}_f.$ Then any perturbation $\Delta A$ such that $\mathcal{X}_c, \mathcal{X}_f$ remain complementary invariant subspaces of $A+\Delta A$ if $$(A+\Delta A) X_a=X_a \hat{\Lambda}_a\,\, \mbox{and} \,\, (A+\Delta A)X_f=X_f\hat{\Lam}_f$$ for some $\hat{\Lam}_a\in\K^{p\times p}$ and $\hat{\Lam}_f\in\K^{(n-p)\times (n-p)}.$ For any arbitrary $\hat{\Lam}_a$ and $\hat{\Lam}_f,$ $\Delta A$ must satisfy $$\Delta AX_a=X_a\hat{\Lam}_a - AX_a, \,\, \mbox{and} \,\, \Delta AX_f=X_f\hat{\Lam}_f - AX_f,$$ which further implies $\Delta A X= \bmatrix{X_a\hat{\Lam}_a - AX_a & X_f\hat{\Lam}_f - AX_f}$ where $X=\bmatrix{X_a & X_f}$ which is a nonsingular matrix. Then the desired result follows by using equation (\ref{eqn:ls}). \hfill{$\square$}
\end{proof}

Moreover, if $\mathfrak{R}(X)=\mathcal{X}$ is an invariant subspace with $AX=XD$ where $X$ is a full rank matrix, then $\sig(A|\mathcal{X})=\sig(D)$ where $\sig(A|\mathcal{X})$ denotes the spectrum of the restriction of the matrix $A$ to the subspace $\mathcal{X}.$ Then we have the following observation from Theorem \ref{Thm:uns1}. If $\mathcal{X}_c$ is an invariant subspace of $A\in\K^{n\times n}$ with $AX_c=X_c \Lam_c, X_c\in\K^{n\times p},\, \Lam_c\in\K^{p\times p}$ then a set of perturbations $\Delta A\in\K^{n\times n}$ such that $\sig(A+\Delta A|\mathcal{X}_c)=\sig(A|\mathcal{X}_c)$ is given by $$\Delta A = (X_aP\Lam_c P^{-1} - AX_a)X_a^{\dagger} + Z(I-X_aX_a^\dagger)$$ where $P\in \K^{p\times p}$ is a nonsingular matrix, $X_a\in\K^{n\times p}$ is a full rank matrix such that $\mathfrak{R}(X_a)=\mathcal{X}_c,$ and $Z\in\K^{n\times n}$ is arbitrary. Indeed, assuming $(A+\Delta A)X_a=X_a(P\Lam_c P^{-1})$ the desired result follows.


\subsection{Structured case}

In this section we determine no spillover structured perturbations which reproduce a desired invariant subspace and preserve a given invariant subspace of a structured matrix. First, we briefly review the properties of invariant subspaces of a matrix $A\in\S$ where $\S\in\{\L, \J\}.$ Here, as defined in Section \ref{sec:1}, $\L$ and $\J$ are Lie and Jordan algebra defined by a scalar product $\langle\cdot,\cdot\rangle_H$ on $\K^n$ where $H$ is a unitary matrix and $H^{\star}= \epsilon_1 H,$ $\epsilon_1\in\{-1, 1\},$ $\star\in \{T,*\}.$ Besides, $A^{[\star]}=H^{-1} A^{\star}H=\epsilon_2 A,\, \star\in\{T, *\},$ where $\epsilon_2=1$ if $A\in \S=\J$, and $\epsilon_2=-1$ if $A\in \S=\L.$ Then we have the following proposition.


\begin{proposition} \label{PropA}
Let $A \in \S\subset\K^{n\times n}.$ Suppose $(X_j,\Lambda_j) \in \K^{n \times p_j} \times \K^{p_j \times p_j},\,j=1,2$ are  invariant pairs of $A.$ Then: 
\begin{itemize}
\item[(a)] $X_j^{\star}HAX_k=\epsilon_2 \Lambda_j^{\star} X_j^{\star}HX_k=X_j^{\star}HX_k\Lambda_k$ where $j,k \in \{1,2\},$
\item[(b)] $X_j^{\star}HX_k=0$ whenever $\sigma (\epsilon_2\Lambda_j^{\star}) \cap \sigma (\Lambda_k) =\emptyset.$
\end{itemize}
\end{proposition}

\begin{proof} Suppose $AX_j=X_j \Lambda_j$ and $AX_k=X_k \Lambda_k$ where $j,\,k\in \{1,2\}.$ Then premultiplying the previous equation by $H$ and operating $\star$ on both sides, $X_j^{\star}A^{\star}H=\Lambda_j^{\star}X_j^{\star}H$. Further postmultiplying it by $X_k$ and using $A^{\star}H= \epsilon_2 HA$ we obtain $X_j^{\star}HAX_k= \epsilon_2 \Lambda_j^{\star}X_j^{\star}HX_k.$ Finally, premultiplying $AX_k=X_k \Lambda_k$ by $X_j^{\star}H$ it follows that $X_j^{\star}HAX_k=X_j^{\star}HX_k\Lambda_k,$ and hence $(a)$ follows. Solving the Sylvester equation $\epsilon_2 \Lambda_j^{\star} X_j^{\star}HX_k-X_j^{\star}HX_k\Lambda_k=0,$ the desired result in $(b)$ follows.  $\hfill{\square}$
\end{proof}


We recall the following result from \cite{AdhAlam}.
\begin{theorem}[Theorem 3.1 and Theorem 3.3, \cite{AdhAlam}]\label{Thm1:AdhAlam}
Let $(X, B)\in \K^{n\times p} \times \K^{n\times p}.$ Then there exists a matrix $A\in\S$ such that $AX=B$ if and only if $BX^\dagger X=B$ and $X^\star HB=\epsilon_1\epsilon_2(X^\star HB)^\star.$ If this condition is satisfied then any such matrix $A$ is given by\footnote{Here we mention that there is a typo in [Section 3, Page 4, \cite{AdhAlam}], the correct expression of $\mathcal{F}_*(X,B)$ should be $BX^\dagger - (BX^\dagger)^* + X^\dagger(X^*B)^*X^\dagger$ if $(X^*B)^*=-X^*B.$ }
 \begin{equation} \label{invrA} 
A =BX^\dagger + \epsilon_1\epsilon_2H^{-1} \left[ (HBX^\dagger)^\star - (X^\dagger)^\star (X^\star HB)^\star X^\dagger\right]  + H^{-1} (I-XX^\dagger)^\star Z(I-XX^\dagger) \end{equation}
where $Z\in\K^{n\times n}$ such that $Z^\star=\epsilon_1\epsilon_2 Z.$ Further, the unique minimal perturbation $A_o\in \S$ for which $\|A_o\|_F = \min\{\|A\|_F : A\in \S,\, AX=B\}$ is given by $$A_o=BX^\dagger + \epsilon_1\epsilon_2H^{-1} \left[ (HBX^\dagger)^\star - (X^\dagger)^\star (X^\star HB)^\star X^\dagger\right]$$ which is obtained by setting $Z=0$ in equation (\ref{invrA}).
\end{theorem}

First we determine all structured perturbations of a structured matrix for which a subspace becomes an invariant subspace of the updated matrix.
\begin{proposition}\label{invrDeltaA}
Let $A\in \S \subset \K^{n \times n}$ and $\mathcal{X}_a$ be a subspace of $\K^n$ of dimension $p$. Then $\mathcal{X}_a$ is an invariant subspace of $A+\Delta A$ for some $\Delta A\in \S$ if and only if there exists a matrix $\Lam_a\in \K^{p \times p}$ such that $X_a^\star HX_a \Lam_a=\epsilon_2 \Lam_a^\star X_a^\star HX_a$ where $X_a\in \K^{n \times p}$ is a full column rank matrix with $\mathfrak{R}(X_a)=\mathcal{X}_a.$ If such a matrix $\Lam_a$ exists then any such perturbation $\Delta A\in \S$ is given by 
\beano
\Delta A&=&(X_a\Lam_a-AX_a)X_a^{\dagger}+\epsilon_1\epsilon_2H^{-1} [\left(H(X_a\Lam_a-AX_a)X_a^{\dagger} \right)^\star-(X_a^{\dagger})^\star (X_a^{\star}HX_a\Lam_a\\&&-X_a^{\star}HAX_a)^\star X_a^{\dagger} ]+H^{-1}(I-X_aX_a^{\dagger})Z(I-X_aX_a^{\dagger})
\eeano 
where $Z\in \K^{n \times n}$ with $Z^\star=\epsilon_1 \epsilon_2Z$ such that $(A+\Delta A)X_a=X_a\Lam_a$. In particular, setting $Z=0,$ the matrix $\Delta A=\Delta A_o$ provides the minimal perturbation. 

 
\end{proposition}

\begin{proof}
Note that $\mathcal{X}_a$ is an invariant subspace of $A+\Delta A$ if $(A+\Delta A)\mathcal{X}_a \subset \mathcal{X}_a$, that is if there exist a matrix $\Lambda_a\in \K^{p \times p}$ such that $(A+\Delta A)X_a=X_a\Lambda_a.$ Therefore $\Delta A$ must satisfy 
\begin{equation}\label{invDelA}
\Delta AX_a=X_a\Lam_a-AX_a.
\end{equation}
From Theorem \ref{Thm1:AdhAlam} it follows that equation $(\ref{invDelA})$ has a solution $\Delta A\in \S$ if and only if \begin{eqnarray}
\label{invr_cond1}
\left(X_a\Lam_a-AX_a \right)X_a^{\dagger}X_a=\left(X_a\Lam_a-AX_a \right) \,\,\mbox{and}\\
\label{invr_cond2}
\left(X_a^\star HX_a\Lam_a-X_a^\star HAX_a \right)=\epsilon_1 \epsilon_2 \left(X_a^\star HX_a\Lam_a-X_a^\star HAX_a \right)^\star.
\end{eqnarray}
As $X_a\in \K^{n \times p}$ is a full column rank matrix so $X_a^{\dagger}=(X_a^\star X_a)^{-1}X_a^\star$ hence condition $(\ref{invr_cond1})$ holds. Further $A\in \S$ that is $(HA)=\epsilon_1\epsilon_2(HA)^\star$ so condition $(\ref{invr_cond2})$ holds if and only if $X_a^\star HX_a\Lam_a=\epsilon_2\Lam_a^\star X_a^\star HX_a$. The rest follows from Theorem \ref{Thm1:AdhAlam}.$\hfill{\square}$
\end{proof}

Now we determine structured perturbations of a Jordan or Lie algebra structured matrix under which an invariant subspace of the matrix is preserved. In addition, under certain assumptions, no spillover perturbations are obtained which preserve complementary invariant subspaces of the matrix, one of which need not be known. 

\begin{theorem}\label{updateA}
Suppose $\S\in\{\L, \J\}$ and $\star\in\{T, *\}.$ Let $A\in\S\subset\K^{n\times n}.$ Suppose $\mathcal{X}_c$ is an invariant subspace of $A$ of dimension $p\leq n$ such that $AX_c=X_c\Lam_c$ where $X_c\in\K^{n\times p}$ is a full rank matrix with $\mathfrak{R}(X_c)=\mathcal{X}_c$ and $\Lam_c\in\K^{p\times p}.$ Then $\mathcal{X}_c$ remains an invariant subspace of $A+\Delta A$ for some matrix $\Delta A\in\S$ if and only if there exist a matrix $\Lam_a\in\K^{p\times p}$ such that $$R^\star (X_c^\star HX_c)R\Lam_a=\epsilon_2\Lam_a^\star R^\star (X_c^\star HX_c) R$$ for some nonsingular matrix $R\in\K^{p\times p}.$  

If such a matrix $\Lam_a$ exists then any such perturbation $\Delta A\in \S$ is given by 
\beano \Delta A &=&  X_c\widetilde{R} (X_cR)^\dagger  +\epsilon_1\epsilon_2H^{-1}\left[(HX_c \widetilde{R} (X_cR)^\dagger)^\star - ((X_cR)^\dagger)^\star(R^\star X_c^\star HX_c  \widetilde{R})^\star (X_cR)^\dagger\right] \\ &&  + H^{-1} \left(I-(X_cR)(X_cR)^\dagger\right) Z \left(I-(X_cR)(X_cR)^\dagger\right)\eeano such that $(A+\Delta A)(X_cR)=(X_cR)\Lam_a,$ where $Z\in\K^{n\times n}$ satisfies $Z^\star=\epsilon_1\epsilon_2 Z$ and $\widetilde{R}=R\Lam_a-\Lam_c R.$ 
\end{theorem}

\begin{proof}
Note that $\mathcal{X}_c$ is an invariant subspace of $A+\Delta A$ for some $\Delta A\in\S$ if there exists a full rank matrix $X_a\in\K^{n\times p}$ such that $(A+\Delta A)X_a=X_a\Lam_a$ where $\mathfrak{R}(X_a)=\mathcal{X}_c$ and $\Lam_a\in\K^{p \times p}.$ Then $X_a=X_cR$ for some nonsingular matrix $R\in\K^{p \times p}$ since $\mathfrak{R}(X_c)=\mathcal{X}_c.$ Consequently, $AX_a=AX_cR=X_c\Lam_cR.$ Then the desired $\Delta A$ has to satisfy  \begin{equation}\label{eqn:adh}\Delta AX_cR=X_c(R\Lam_a -\Lam_c R).\end{equation} Then by Theorem \ref{Thm1:AdhAlam} there exists a $\Delta A\in\S$ which satisfies the equation (\ref{eqn:adh}) if and only if \begin{equation} \label{eqn:cond22} R^\star (X_c^\star HX_c)(R\Lam_a-\Lam_c R)=\epsilon_1\epsilon_2\left(R^\star (X_c^\star HX_c)(R\Lam_a-\Lam_c R)\right)^\star. \end{equation} 

Moreover, since $A\in \S$ and $AX_c=X_c\Lam_c$ so by part $(a)$ of Proposition \ref{PropA} we have $X^{\star}_cHX_c \Lam_c=\epsilon_1\epsilon_2 \left(X^{\star}_cHX_c \Lam_c\right)^\star$, and hence the condition $(\ref{eqn:cond22})$ reduces to, \begin{equation} R^\star X_c^\star HX_cR\Lam_a=\epsilon_2\Lam_a^{\star}R^{\star} X_c^\star HX_cR \end{equation} and therefore the desired result follows. \hfill{$\square$}
\end{proof}

Next we determine no spillover structured perturbations which preserve complementary invariant subspaces of a structured matrix one of which need not be known.

\begin{theorem} \label{updateA_nospillover}
Let $\mathcal{X}_c$ be an invariant subspace of a matrix $A\in \S\subset \K^{n\times n}$ of dimension $p$. Let $X_c$ be a full rank matrix such that $AX_c=X_c\Lam_c$ and $\mathfrak{R}(X_c)=\mathcal{X}_c$ with $\Lam_c\in \K^{p \times p}$. Choose a matrix $\Lam_a\in\K^{p \times p}$ such that $X_c^{\star}HX_c \Lambda_a=\epsilon_1\epsilon_2 (X_c^{\star}HX_c \Lambda_a)^\star.$ Then a perturbation $\Delta A\in \S$ such that the complementary invariant subspaces $\mathcal{X}_c$ and $\mathcal{X}_f$ of $A$ are preserved as complementary invariant subspaces of $A+\Delta A,$ is given by  $$\Delta A = BX^{-1} + \epsilon_1\epsilon_2 H^{-1}\left[(HBX^{-1})^\star - (X^{-1})^\star(X^\star HB)^\star X^{-1} \right]$$ where $X=\bmatrix{X_c & X_f}, B=\bmatrix{X_c\Lam_a - AX_c & 0},$  if there exists a  matrix $\Lam_f\in \K^{(n-p)\times (n-p)}$ such that $AX_f=X_f\Lam_f,$ $\sigma(\epsilon_2 \Lambda_c^{\star}) \cap \sigma(\Lambda_f)=\emptyset,$ and $\mathfrak{R}(X_f)=\mathcal{X}_f.$ Thus $$(A+\Delta A)\bmatrix{X_c & X_f}=\bmatrix{X_c & X_f} \bmatrix{\Lam_a & 0 \\ 0 & \Lam_f}.$$

In addition, if $X_f$ and $\Lam_f$ are not known but the existence of these matrices are assumed then the complementary invariant subspaces $\mathcal{X}_c$ and $\mathcal{X}_f$ of $A$ remain invariant for the perturbed matrix $A+\Delta A\in \S$ where the no spillover perturbation of rank less equal to $p$ is given by
 $$\triangle A=X_c(\Lambda_a-\Lambda_c)(X_c^{\star}HX_c)^{-1}X_c^{\star}H.$$

\end{theorem}

\begin{proof}
Let $\Lam_f \in \K^{(n-p)\times (n-p)}$ be a matrix such that $AX_f=X_f\Lam_f$  and $\sigma(\epsilon_2 \Lambda_c^{\star}) \cap \sigma(\Lambda_f)=\emptyset.$ Then by Proposition \ref{PropA} (b) it follows that $X_c^\star HX_f=0.$ Then we show that there exists a matrix $\Delta A\in \S$ such that $(A+\Delta A)X_f=X_f\Lam_f$ and $(A+\Delta A)X_c=X_c\Lam_a,$ that is, $$\Delta A \bmatrix{X_c & X_f} = \bmatrix{X_c\Lam_a - AX_c & 0}.$$ From Theorem \ref{Thm1:AdhAlam}, such a $\Delta A$ exists since \beano  \bmatrix{X_c & X_f}^\star H  \bmatrix{X_c\Lam_a - AX_c & 0}  &=& \bmatrix{X_c^\star HX_c\Lam_a - X_c^\star HAX_c & 0 \\ X_f^\star HX_c \Lam_a - X_f^\star HAX_c &0} \\ &=& \bmatrix{X_c^\star HX_c\Lam_a- X_c^\star HX_c\Lam_c & 0 \\ 0 & 0} \\ &=& \epsilon_1\epsilon_2 \left(\bmatrix{X_c^\star HX_c\Lam_a- X_c^\star HX_c\Lam_c & 0 \\ 0 & 0}\right)^\star \\ &=& \epsilon_1\epsilon_2\left( \bmatrix{X_c & X_f}^\star H  \bmatrix{X_c\Lam_a - AX_c & 0}\right)^\star,\eeano where the second last step follows from the assumption $X_c^\star HX_c\Lam_a=\epsilon_1\epsilon_2(X_c^\star HX_c\Lam_a)^\star,$ and by Proposition \ref{PropA} (a), which gives $X_c^\star HX_c\Lam_c=\epsilon_1\epsilon_2(X_c^\star HX_c\Lam_c)^\star.$ 

However, if the matrices $X_f$ and $\Lam_f$ are not known then assuming the existence of such matrices, a required structured no spillover perturbation is determined as follows. Let $\Delta A=\hat{A}X_c^\star H$ for some $\hat{A}\in\K^{n\times p}.$ Then it obviously follows that $\Delta AX_f=0=AX_f-X_f\Lam_f.$ The matrix $\hat{A}$ can be found by solving the matrix equation \begin{equation} \label{ultimate_eqn}
 \hat{A} X_c^{\star}HX_c=X_c\Lam_a-AX_c=X_c(\Lambda_a-\Lambda_c) 
\end{equation} such that $\Delta A=\hat{A}X_c^\star H\in \S.$ Further $\left[X_c\,\,X_f \right],\,H$ are invertible and  $X_c^\star HX_f=0,$ since $\sigma(\epsilon_2 \Lam_c^\star) \cap \sigma(\Lam_f)=\emptyset.$ Hence $\left[X_c\,\,X_f\right]^\star H \left[X_c\,\,X_f\right]=\mbox{diag}\left(X_c^\star HX_c,\,\,X_f^\star HX_f \right)$ is invertible, thus $X_c^\star HX_c$ is nonsingular. Therefore solving equation (\ref{ultimate_eqn}) we obtain $$\Delta A= X_c(\Lambda_a-\Lambda_c)(X_c^{\star}HX_c)^{-1}X_c^\star H.$$ Thus it still needs to prove that $\Delta A^{[\star]}=\epsilon_2 \Delta A.$

Note that, as per the assumption, $(X_c^{\star}HX_c) \Lambda_a=\epsilon_1\epsilon_2 (X_c^{\star}HX_c \Lambda_a)^\star=\epsilon_2 \Lam_a^\star X_c^\star HX_c.$ Further by Proposition \ref{PropA} (a), $(X_c^{\star}HX_c) \Lambda_c= \epsilon_2 \Lam_c^\star X_c^\star HX_c.$ Subtracting these two equations, we obtain \begin{equation}\label{eqn:lalc}(X_c^\star HX_c)^{-1}(\Lam_a^\star - \Lam_c^\star) = \epsilon_2 (\Lam_a-\Lam_c)(X_c^\star HX_c)^{-1}.\end{equation} Then 
\beano (\Delta A)^{[\star]} &=& H^{-1}(\Delta A)^\star H \\ &=& \epsilon_1 H^{-1} \left[H^\star X_c(X_cHX_c)^{-1} (\Lam_a-\Lam_c)^\star X_c^\star \right] H \\ &=& \epsilon_1 H^{-1} \left[\epsilon_2 H^\star X_c(\Lam_a -\Lam_c)(X_c^\star HX_c)^{-1}X_c^\star \right] H, \, \mbox{by equation} \, (\ref{eqn:lalc}) \\ &=& \epsilon_2 H^{-1}HX_c(\Lam_a - \Lam_c)(X_c^\star HX_c)^{-1}X_c^\star H \\ &=& \epsilon_2 \Delta A.\eeano This completes the proof. \hfill{$\square$}
\end{proof}

We emphasize that the spectral condition  $\sigma(\epsilon_2 \Lambda_c^{\star}) \cap \sigma(\Lambda_f)=\emptyset$ is not a strong condition to be satisfied by $\Lam_f$ and $\Lam_c.$ Note that $\sigma(A)=\sigma(\Lambda_c) \cup \sigma(\Lambda_f).$ Further, from Proposition \ref{eigvec_ortho} (a), it follows that eigenvalues of a structured matrix $A\in\S$ occur in pairs $(\lam, \epsilon_2 \lam^\star)$. Thus the spectral condition on $\Lam_f$ and $\Lam_c$ eventually means that each of these matrices should preserve the eigenvalue pairing of $A.$ An important feature of the perturbations $\Delta A$ derived in Theorem \ref{updateA_nospillover} is that $$\sig(A+\Delta A| \mathcal{X}_c)= \Lam_a \neq \sig(A| \mathcal{X}_c) \,\, \mbox{and} \,\, \sig(A+\Delta A| \mathcal{X}_f)= \sig(A| \mathcal{X}_f).$$

\section{Modifying eigenvalues and preserving Jordan chains under structured perturbations}

In this section we consider the problem {\bf (P1)}, that is, we determine structured perturbations of a structured matrix such that the corresponding perturbed matrices reproduce a desired set of eigenvalues and preserve Jordan chains of the unperturbed matrix. We emphasize that it is desirable to find a real structured perturbation $\Delta A$ when the given structured matrix $A$ is real. 

Recall that we consider the space of structured matrices $\S\in\{\J, \L\}$ where $\J$ and $\L$ are the Jordan and Lie algebra defined by a scalar product $\langle\cdot,\cdot\rangle_H$ on $\K^n.$ Here $H$ is a unitary matrix when $\K=\C$, and $H$ is an orthogonal matrix when $\K=\R,$ and $H^\star=\epsilon_1 H,$ $\epsilon_1\in\{1,-1\}$ where $\star=*$ (conjugate transpose) when $\K=\C$ and $\star=T$ (transpose) when $\K=\R.$ Besides, if $A\in\S$ then $A^{[\star]}=\epsilon_2 A$ where $\epsilon_2\in\{1,-1\}.$  Moreover $\epsilon_2=1$ if $\S=\J,$ and $\epsilon_2=-1$ if $\S=\L.$ 



Then we have following proposition about the Jordan pairs of structured matrices. 
\begin{proposition} \label{eigvec_ortho}
Let $A\in \S.$ Then the following are true.
\begin{itemize}
\item[(a)] $\lam$ is an eigenvalue of $A$ if and only if $\epsilon_2 \lambda^{\star}$ is an eigenvalue of $A.$ Indeed the partial multiplicities corresponding to $\epsilon_2 \lambda^{\star}$ are equal to those corresponding to $\lam$. 
\item[(b)] Let $(\lam_1,X_1)$ and $(\lam_2,X_2)$ be Jordan pairs of $A.$ Then $X_1^\star HX_2=0$ if $\lam_2 \neq \epsilon_2 \lam_1^{\star}.$ Thus if $X=\bmatrix{X_1 & X_2}$ then $X^\star HX=\mbox{diag}\left(X_1^\star HX_1,\,X_2^\star HX_2\right).$ In particular, if $(\lam, X)$ is a Jordan pair of $A\in\S$ and $\lam\neq \ep_2\lam^\star$ then $X^\star HX=0.$
\item[(c)] Let $(\lam,X)$ be a Jordan pair of $A$ such that $\lam \neq \ep_2 \lam^\star$ then there must exist another Jordan pair $(\ep_2 \lam^\star ,\tilde{X})$ of $A$ such that $\#(\tilde{X})=\#(X)$. Then taking $X_0=\bmatrix{X& \tilde{X}}$ we have $X_0^\star HX_0=\bmatrix{0& X^\star H\tilde{X} \\ \ep_1 (X^\star H\tilde{X})^\star& 0}$.  
\item[(d)] If $\lam = \ep_2\lam^\star$ is an eigenvalue of $A$ corresponding to right eigenvector $x$ then $Hx$ is the corresponding left eigenvector. However, in this case $\langle x,x\rangle_H$ need not be zero.
\end{itemize}
\end{proposition}

\begin{proof}
The proof of $(a)$ follows from the fact that  $A$ is unitarily similar to $\epsilon_2 A^{\star}.$ As $(\lam_1,X_1)$ and $(\lam_2,X_2)$ are Jordan pairs of $A$ so that $AX_1=X_1J(\lam_1)$ and $AX_2=X_2J(\lam_2)$ where $J(\lam_i)$ denotes the Jordan block corresponding to $\lam_i$ of size $\#(X_i),\,i=1,2$. By Proposition \ref{PropA}$(b)$ it follows that $X_1^\star HX_2=0$ whenever $\lam_2 \neq \epsilon_2 \lam_1^{\star}$ thus $(b)$ follows. Next, as $(\lam,X)$ is a Jordan pair of $A$ so $AX=XJ(\lam)$ where $J(\lam)$ is a Jordan block corresponding to $\lam$ of size $\#(X)$. Since $\lam \neq \ep_2 \lam^\star$ so by Proposition \ref{eigvec_ortho} $(a)$ it follows that $\ep_2\lam^\star$ must be an eigenvalue of $A$ whose partial multiplicity is same as the partial multiplicity of $\lam$ hence there must exist a matrix $\tilde{X}$ with $\#(\tilde{X})=\#(X)$ satisfying $A\tilde{X}=\tilde{X} J(\ep_2 \lam^\star)$. Thus $(\ep_2 \lam^\star,\tilde{X})$ is a Jordan pair of $A$. Further by Proposition \ref{PropA} $(b)$ it follows that $X^\star HX=\tilde{X}^\star H\tilde{X}=0$ as $\lam \neq \ep_2 \lam^\star$. Hence taking $X_0=\bmatrix{X& \tilde{X}}$ we have $X_0^\star HX_0=\bmatrix{0& X^\star H \tilde{X} \\ \ep_1(X^\star H \tilde{X})^\star& 0}$ where the last equality follows by using $\tilde{X}^\star HX=\ep_1(X^\star H \tilde{X})^\star.$ Proof of $(d)$ is as follows. If $Ax=\lam x$ then 
$$ HAx=\lam Hx \Rightarrow \ep_2 A^\star Hx=\lam Hx \Rightarrow (Hx)^\star A=\epsilon_2 \lam^\star (Hx)^\star.$$ This completes the proof.  $\hfill{\square}$
\end{proof}
Then we have the following results on eigenvectors of $A\in \S$.
\begin{corollary}
Let $(\lam_1,x_1)$ and $(\lam_2,x_2)$ be eigenpairs of $A\in \S.$ Then $\langle x_1,x_2\rangle_H=0$ if $\lam_2 \neq \epsilon_2 \lam_1^{\star}.$ Thus if $X=\bmatrix{x_1 & x_2}$ then $X^\star HX=\mbox{diag}\left(x_1^\star Hx_1,\,x_2^\star Hx_2\right).$ In particular, if $(\lam, x)$ is an eigenpair of $A\in\S$ and $\lam\neq \ep_2\lam^\star$ then $\langle x,x\rangle_H=0.$
\end{corollary}
\begin{corollary} \label{blockform}
Let $(\lam,x)$ be an eigenpair of $A \in \S$ such that $\lam \neq \epsilon_2 \lam^{\star}$, then by part $(a)$ of Proposition $\ref{eigvec_ortho},$ there exists another eigenvalue $\epsilon_2 \lam^{\star}$ of $A$ with corresponding eigenvector $\tilde{x}$. Taking $X_0=[x \,\, \tilde{x}]$ we obtain $$X_0^{\star}HX_0=\bmatrix{0& h \\ \epsilon_1 h^{\star}& 0}$$ since  by $(b)$ of Proposition \ref{PropA} we have $\langle x, x\rangle_H=\langle \tilde{x}, \tilde{x} \rangle_H=0$ and $\langle x, \tilde{x}\rangle_H=\epsilon_1 (\langle \tilde{x}, x \rangle_H)^{\star}=\epsilon_1 h^{\star}$ where $h=\langle \tilde{x}, x \rangle_H.$ In particular $X^{\star}_0HX_0$ is nonsingular when $h=\langle \tilde{x}, x \rangle_H \neq 0.$
\end{corollary}

Now we have the following theorem which describes the choice of structured perturbations for a given structured matrix such that the perturbed matrices reproduce a set of desired eigenvalues while keeping the Jordan chains of the unperturbed matrix invariant. We know that the number of Jordan chains corresponding to an eigenvalue is the geometric multiplicity of the eigenvalue. Let $\lam^c_i$ be an eigenvalue of $A$ having $m_i$ number of Jordan chains then it implies that there exist $m_i$ number of Jordan pairs of $A$ corresponding to $\lam^c_i$ and we denote them as $(\lam^c_i,X^c_{i,l}),\, l=1,\hdots, m_i$. Suppose $\alpha(\lam^c_i)$ denotes the algebraic multiplicity of an eigenvalue $\lam^c_i$ then it follows that $\alpha(\lam^c_i)=\sum_{l=1}^{m_i} \#(X^c_{i,l}).$ We recall that $(\lam^c_i,X^c_{i,l})$ said to be a Jordan pair of $A$ if it satisfies $AX^c_{i,l}=X^c_{i,l}J_{l}(\lam^c_i)$ where $J_{l}(\lam^c_i)$ denotes the Jordan block corresponding to $\lam^c_i$ of size $\#(X^c_{i,l}).$ 


\begin{theorem} \label{P1soln_comp}
Let $A\in\S\subset\C^{n\times n}.$ Suppose $\{(\lam^c_i,X^c_{i,l})\}_{l=1}^{m_i}$, $\{(\ep_2 \overline{\lam^c_i},\tilde{X}^c_{i,l})\}_{l=1}^{m_i}$, $\{(\lam^c_j= \ep_2 \overline{\lam^c_j},X^c_{j,l})\}_{l=1}^{m_j}$ are Jordan pairs of $A$, where $\lam^c_i \neq \ep_2 \overline{\lam^c_i},$ $i=1, \hdots, q,$ $j=2q+1, \hdots, p$. Let $\lam^a_i, \lam^a_j$ be a collection of scalars with $\lam^a_i\neq \ep_2 \overline{\lam^a_i}$ and $\lam^a_j = \ep_2 \overline{\lam^a_j},$  $i=1,\hdots, q,$ $j=2q+1, \hdots, p.$ Then any structured perturbation $\Delta A\in\S$ of $A$ for which  $\{(\lam^a_i,X^c_{i,l})\}_{l=1}^{m_i},\, \{(\ep_2 \overline{\lam^a_i}, \tilde{X}^c_{i,l})\}_{l=1}^{m_i},$ $\{(\lam^a_j, X^c_{j,l})\}_{l=1}^{m_j}$ are Jordan pairs of $A+\Delta A$ is given by 
 \beano \Delta A &=& X_c(\Lam_a-\Lam_c)X_c^{\dagger}+\ep_2 H^{-1}(X_c^\dagger)^*(\Lam_a-\Lam_c)^*X_c^*H-H^{-1}(X_c^{\dagger})^*X_c^*HX_c(\Lam_a-\Lam_c)X_c^{\dagger} \\ && +H^{-1}(I-X_cX_c^{\dagger})Z(I-X_cX_c^{\dagger}),\eeano where $X_c=\bmatrix{X_1^c & \hdots & X_q^c & X_{2q+1}^c & \hdots & X_p^c},\,$ $X_i^c=\bmatrix{\hat{X}_i^c & \tilde{X}_i^c},\,\hat{X}_i^c=\bmatrix{X^c_{i,1}\hdots X^c_{i,m_i}},\,\tilde{X}^c_i=\bmatrix{\tilde{X}^c_{i,1}\hdots \tilde{X}^c_{i,m_i}},\,X^c_j=\bmatrix{X^c_{j,1}\hdots X^c_{j,m_j}},\, \Lam_c=\diag(\Lam_1^c, \hdots, \Lam_q^c,\, \Lam^c_{2q+1}, \hdots, \Lam_p^c),\Lam^c_i=\diag(\hat{\Lam}^c_i,\tilde{\Lam}^c_i),\linebreak \hat{\Lam}^c_i=\diag(J_{1}(\lam^c_i),\dots,J_{m_i}(\lam^c_i)),\,\tilde{\Lam}^c_i=\diag(J_{1}(\ep_2 \overline{\lam^c_i}),\dots,J_{m_i}(\ep_2 \overline{\lam^c_i})),\,\Lam^c_j=\diag(J_{1}(\lam^c_j),\linebreak\hdots,J_{m_j}(\lam^c_j)),$ $\Lam_a=\diag(\Lam_1^a, \hdots, \Lam_q^a,\, \Lam^a_{2q+1}, \hdots, \Lam_p^a),\,\Lam^a_i=\diag(\hat{\Lam}^a_i,\tilde{\Lam}^a_i),\,\hat{\Lam}^a_i=\diag(J_{1}(\lam^a_i),\linebreak \dots,J_{m_i}(\lam^a_i)),\,\tilde{\Lam}^a_i=\diag(J_{1}(\ep_2 \overline{\lam^a_i}),\dots,J_{m_i}(\ep_2 \overline{\lam^a_i})),\,\Lam^a_j=\diag(J_{1}(\lam^a_j),\hdots,J_{m_j}(\lam^a_j))$ with \linebreak $\#J_l(\lam_i^a)=\#J_l(\lam_i^c),\, l=1,\hdots, m_i,$ $\#J_l(\lam_j^a)=\#J_l(\lam_j^c),\, l=1,\hdots, m_j$ and $Z^*=\ep_1\ep_2 Z\in\C^{n\times n}$ is arbitrary.

Let $\lam^f_k,\, k=r+1, \hdots, n$ be the rest of the eigenvalues of $A$ with $r=\sum_{i=1}^{q} 2 \alpha (\lam^c_i)+\sum_{j=2q+1}^{p} \alpha (\lam^c_j)$. If $$\left\{\lam^c_i,\,\ep_2 \overline{\lam^c_i}, \lam^c_j : i=1,\dots,q, j=2q+1, \hdots ,p \right\}\cap \{\lam^f_k\,:\,k=r+1,\dots,n\}=\emptyset$$ 
then a no spillover structured perturbation of rank $r$ is given by
\begin{equation}\label{eqn:nospill} \Delta A=X_c(\Lam_a-\Lam_c)(X_c^*HX_c)^{-1}X_c^*H.\end{equation} 
 

\end{theorem}
\begin{proof}
It is given that $AX^c_i=X^c_i\Lam^c_i,\,AX^c_j=X^c_j\Lam^c_j$ holds and as $\lam^c_i,\,\ep_2 \overline{\lam^c_i},\,\lam^c_j$ are all different so from Proposition \ref{PropA} $(b)$ we have $X_c^*HX_c=\diag((X^c_1)^*HX^c_1,\dots,(X^c_q)^*HX^c_q,\,(X^c_{2q+1})^*HX^c_{2q+1},\linebreak \hdots,(X^c_p)^*HX^c_p).$ Further by Proposition \ref{PropA} $(b)$ it implies that $(\hat{X}^c_i)^*H\hat{X}^c_i=(\tilde{X}^c_i)^*H\tilde{X}^c_i=0$ as $\lam^c_i \neq \ep_2 \overline{\lam^c_i},$ thus $(X^c_i)^*HX^c_i=\bmatrix{0& (\hat{X}^c_i)^*H\tilde{X}^c_i \\ \ep_1 ((\hat{X}^c_i)^*H\tilde{X}^c_i)^*& 0}$. Therefore by a simple computation it follows that $X_c^*HX_c(\Lam_a-\Lam_c)=\ep_1\ep_2 (X_c^*HX_c(\Lam_a-\Lam_c))^*$. 
Hence setting $R$ as the identity matrix the condition in equation $(\ref{eqn:cond22})$ is achieved and the desired expression for the matrix \beano \Delta A &=& X_c(\Lam_a-\Lam_c)X_c^{\dagger}+\ep_2 H^{-1}(X_c^\dagger)^*(\Lam_a-\Lam_c)^*X_c^*H-H^{-1}(X_c^{\dagger})^*X_c^*HX_c(\Lam_a-\Lam_c)X_c^{\dagger} \\ && +H^{-1}(I-X_cX_c^{\dagger})Z(I-X_cX_c^{\dagger})\in \S,\eeano follows from Theorem \ref{updateA} with $Z^*=\ep_1 \ep_2 Z \in \C^{n \times n}$ is arbitrary. 

Further, since $\left\{\lam^c_i,\,\ep_2 \overline{\lam^c_i},\,\lam^c_j\,:\,i=1, \hdots, q,\,j=2q+1, \hdots, p\right\}\cap \{\lam^f_k\,:\,k=r+1, \hdots, n\}=\emptyset$ then the existence of structured no spillover perturbation is guaranteed by Theorem \ref{updateA_nospillover} and the structured no spillover perturbation is given by $\Delta A=X_c(\Lam_a-\Lam_c)(X_c^*HX_c)^{-1}X_c^*H \in \S. \hfill{\square}$
\end{proof}

The following theorem presents the structured solution to problem \textbf{(P1)} when all the eigenvalues $\lam_i^c,$ which are to be changed, are simple.

\begin{theorem} \label{P1_comp_distinct}
Let $A\in\S\subset\C^{n\times n}.$ Suppose $(\lam^c_i,x^c_i),\, (\ep_2 \overline{\lam^c_i},\tilde{x}^c_i), (\lam^c_j= \ep_2 \overline{\lam^c_j},x^c_j)$ are eigenpairs of $A$ where $i=1:q,\, j=2q+1:p$, and $\lam^c_i\neq \ep_2 \overline{\lam^c_i}.$ Let $\lam^a_i,\, \lam^a_j$ be a collection of scalars with $\lam^a_i\neq \ep_2 \overline{\lam^a_i}$ and $\lam^a_j = \ep_2 \overline{\lam^a_j}$ where  $i=1:q, \,j=2q+1:p.$ Suppose the eigenvalues $\lam_i^c,\,\epsilon_2 \overline{\lam^c_i},\,\lam_j^c$ are all simple and distinct. Then any structured perturbation $\Delta A\in\S$ of $A$ for which  $(\lam^a_i,x^c_i),\, (\ep_2 \overline{\lam^a_i}, \tilde{x}^c_i), (\lam^a_j, x^c_j)$ are eigenpairs of $A+\Delta A$ is given by 
\beano \Delta A &=& X_c(\Lam_a-\Lam_c)X_c^{\dagger}+\ep_2 H^{-1}(X_c^\dagger)^*(\Lam_a-\Lam_c)^*X_c^*H-H^{-1}(X_c^{\dagger})^*X_c^*HX_c(\Lam_a-\Lam_c)X_c^{\dagger} \\ && +H^{-1}(I-X_cX_c^{\dagger})Z(I-X_cX_c^{\dagger}),\eeano where $X_c=\bmatrix{X_1^c & \hdots & X_q^c & x_{2q+1}^c & \hdots & x_p^c},\,X_i^c=\bmatrix{x_i^c & \tilde{x}_i^c},$ $\Lam_c=\diag(\Lam_1^c, \hdots, \Lam_q^c,\linebreak \lam^c_{2q+1}, \hdots, \lam_p^c),$ $\Lam_i^c =\diag(\lam_i^c, \ep_2\overline{\lam_i^c}),$ $\Lam_a=\diag(\Lam_1^a, \hdots, \Lam_q^a,\, \lam^a_{2q+1}, \hdots, \lam_p^a),$ $\Lam_i^a =\diag(\lam_i^a, \ep_2 \overline{\lam_i^a}),$ and $Z^*=\ep_1\ep_2 Z\in\C^{n\times n}$ is arbitrary.
 
Let $\lam^f_k, k=p+1, \hdots, n$ be the rest of the eigenvalues of $A.$ If $$\left\{\lam^c_i,\,\ep_2 \overline{\lam^c_i}, \lam^c_j : i=1,\dots,q, j=2q+1, \hdots ,p \right\}\cap \{\lam^f_k\,:\,k=p+1,\dots,n\}=\emptyset$$ 
then a no spillover structured perturbation of rank $p$ is given by
\begin{equation}\label{eqn:nospill2}\Delta A=X_c(\Lam_a-\Lam_c)(X_c^*HX_c)^{-1}X_c^*H.\end{equation}
\end{theorem} 
\begin{proof}
It is given that $AX^c_i=X^c_i\Lam^c_i$ and $Ax^c_j=\lam^c_j x^c_j$ holds. Since $\lam_i^c,\,\ep_2 \overline{\lam_i^c},\, \lam_j^c$  are all distinct so using Proposition \ref{PropA} $(b)$ and applying Corollary \ref{blockform} we obtain $X_c^*HX_c=\diag(\Gamma_1,\dots, \Gamma_q,\,(x^c_{2q+1})^*Hx^c_{2q+1},\dots,(x^c_p)^*Hx^c_p)$ with $\Gamma_i=\bmatrix{0& (x^c_i)^*H\tilde{x}^c_i \\ \ep_1((x^c_i)^*H\tilde{x}^c_i)^*& 0}$. Hence it follows that $X_c^*HX_c(\Lam_a-\Lam_c)=\ep_1 \ep_2(X_c^*HX_c(\Lam_a-\Lam_c))^*$ holds. Therefore the desired result follows from Theorem \ref{P1soln_comp}. $\hfill{\square}$

\end{proof}

\begin{remark}
Note that, a sufficient condition \cite{savchenko2004change} for preservation of an eigenvalue $\lam$ of a matrix $A$ under perturbation is that the geometric multiplicity of $\lam$ should be greater than the rank of the perturbation. However, Theorem \ref{P1soln_comp} shows that this need not be a necessary condition for preserving eigenvalues of structured matrices under structured perturbation. Notice that all the eigenvalues $\lam^f_k$ that are preserved under the no spillover perturbation in equation $(\ref{eqn:nospill})$ is independent of geometric multiplicities of the $\lam_k^f.$ For example consider a symmetric matrix 
$$A=\bmatrix{-0.69970&  -1.43911&   0.76575 \\
  -1.43911&   1.46812&   2.08426 \\
   0.76575&   2.08426&   2.10423}.$$ Taking $H$ as the identity matrix of order $3$ it follows that $A\in \J.$ Then $\lam^c_1=-2.1246,\,\lam^c_2=1.0711,\,\lam^f_3=3.9262$ are eigenvalues of $A$. Suppose we wish to replace the eigenvalues $\lam^c_1,\,\lam^c_2$ of $A$ by the desired numbers $\lam^a_1= 2.1457,\,\lam^a_2=1.3342$ in such a way that $\lam^a_1,\,\lam^a_2,\,\lam^f_3$ becomes the eigenvalues of $A+\Delta A,$ that is $\lam^f_3$ remains the eigenvalue of $A$ and $A+\Delta A$ for some perturbation $\Delta A$. Then taking 
\begin{center}
$\Lam_c=\diag(-2.1246,\,1.0711),\,
X_c=\bmatrix{-0.74904&   0.65664\\
  -0.53038&  -0.51457\\
   0.39704&   0.55141},\,\Lam_a=\diag(2.1457,\,1.3342)$
\end{center}
and using the expression in equation $(\ref{eqn:nospill})$ we obtain a no spillover structured perturbation 
$$\Delta A=\bmatrix{2.50934&   1.60757&  -1.17472 \\
   1.60757&   1.27089&  -0.97390 \\
  -1.17472&  -0.97390&   0.75318 }.$$ 
Then we have verified that $\lam^a_1,\,\lam^a_2,\,\lam^f_3$ are eigenvalues of $A+\Delta A,$ that is $\lam^f_3$ is the eigenvalue of $A$ and $A+\Delta A$. Also we note that $rank(\Delta A)=2$ whereas geometric multiplicity of $\lam^f_3$ is $1$. Therefore the given condition in \cite{savchenko2004change} for preservation of eigenvalue is not a necessary condition.
\end{remark}

In the next theorem we present the real structured solution to problem \textbf{(P1)} for structured matrix $A\in \L \subset \R^{n \times n}$ with real $H=\ep_1H^T$. For a matrix $Z=[z_{ij}]\in\C^{n\times n}$ we denote $\overline{Z}=.[\overline{z}_{ij}].$

\begin{theorem} \label{P1soln_real_L}
Let $A\in\L\subset\R^{n\times n}.$ Suppose  $\lam_j^c\in\C\smallsetminus (\R\cup i\R),$ is a nonzero eigenvalue of $A$ and $\{(\lam^c_j,X^c_{j,l})\}_{l=1}^{m_j}$, $\{(\overline{\lam^c_j},\overline{X}^c_{j,l})\}_{l=1}^{m_j}$, $\{(-\lam^c_j,\tilde{X}^c_{j,l})\}_{l=1}^{m_j}$, $\{(-\overline{\lam^c_j},\overline{\tilde{X}}^c_{j,l})\}_{l=1}^{m_j}$ are the corresponding Jordan pairs of $A,$ $j =1,\hdots ,r_1.$ Let $\lam^c_k\in i\R$ be a nonzero eigenvalue of $A$ and $\{(\lam^c_k,X^c_{k,l})\}_{l=1}^{m_k},$ $\{(\overline{\lam^c_k},\overline{X}^c_{k,l})\}_{l=1}^{m_k}$ are the Jordan pairs associated with $\lam^c_k, k=r_1+1, \hdots, r_2.$ Let $\lam_r^c\in\R$ be a nonzero eigenvalue of $A$ and $\{(\lam^c_r,X^c_{r,l})\}_{l=1}^{m_r},$ $\{(-\lam^c_r,\tilde{X}^c_{r,l})\}_{l=1}^{m_r}$ are the associated Jordan pairs of $A,$ $r=r_2+1, \hdots, p.$

 
Let $\lam^a_j,\, \lam^a_k,\,\lam^a_r$ be a collection of scalars with $\lam^a_j \in \C\smallsetminus (\R \cup i\R),\,\lam^a_k\in i\R$ and $\lam^a_r \in \R$ where  $j=1, \hdots, r_1,\, k=r_1+1, \hdots, r_2,\, r=r_2+1, \hdots, p.$ Then any real structured perturbation $\Delta A\in\L$ of $A$ for which $\{(\lam^a_j,X^c_{j,l})\}_{l=1}^{m_j}$, $\{(\overline{\lam^a_j},\overline{X}^c_{j,l})\}_{l=1}^{m_j}$, $\{(-\lam^a_j,\tilde{X}^c_{j,l})\}_{l=1}^{m_j}$, $\{(-\overline{\lam^a_j},\overline{\tilde{X}}^c_{j,l})\}_{l=1}^{m_j}$, $\{(\lam^a_k,X^c_{k,l})\}_{l=1}^{m_k}$, $\{(\overline{\lam^a_k},\overline{X}^c_{k,l})\}_{l=1}^{m_k},$ $\{(\lam^a_r,X^c_{r,l})\}_{l=1}^{m_r},$ $\{(-\lam^a_r,\tilde{X}^c_{r,l})\}_{l=1}^{m_r}$ are Jordan pairs of $A+\Delta A$ is given by 
 \beano \Delta A &=& X_c(\Lam_a-\Lam_c)X_c^{\dagger}- H^{-1}(X_c^\dagger)^*(\Lam_a-\Lam_c)^*X_c^*H-H^{-1}(X_c^{\dagger})^*X_c^*HX_c(\Lam_a-\Lam_c)X_c^{\dagger} \\ && +H^{-1}(I-X_cX_c^{\dagger})Z(I-X_cX_c^{\dagger}),\eeano where $$X_c=\left[X_1^c \hdots X_{r_1}^c \, X_{r_1+1}^c \hdots  X_{r_2}^c \,X^c_{r_2+1} \hdots X^c_p\right],$$ 
 $X_j^c=\left[Y^c_j \, Z_j^c\right],\,Y^c_j=\left[\hat{X}^c_j\, \overline{\hat{X}^c_j}\right],\,Z^c_j=\left[\tilde{X}^c_j\, \overline{\tilde{X}^c_j}\right],\,\hat{X}^c_j=\left[X^c_{j,1}\hdots X^c_{j,m_j} \right],\, \tilde{X}^c_j=\left[\tilde{X}^c_{j,1}\hdots \tilde{X}^c_{j,m_j}\right],\,X^c_k=\left[\hat{X}^c_k\,\,\overline{\hat{X}^c_k}\right],\,\hat{X}^c_k=\left[X^c_{k,1}\dots X^c_{k,m_k}\right], X^c_r=\left[\hat{X}^c_r\,\,\tilde{X}^c_r\right],\,\hat{X}^c_r=\left[X^c_{r,1}\dots X^c_{r,m_r}\right],\,\tilde{X}^c_r=\left[\tilde{X}^c_{r,1}\dots \tilde{X}^c_{r,m_r}\right],$ $$\Lam_c=\diag(\Lam_1^c, \hdots, \Lam_{r_1}^c, \Lam^c_{r_1+1},\hdots,\Lam^c_{r_2},\,\Lam^c_{r_2+1},\hdots ,\Lam^c_p),$$ 
 $\Lam^c_j=\diag(U^c_j,\, V^c_j),\,U^c_j=\diag(\hat{\Lam}^c_j,\overline{\hat{\Lam}^c_j}),$ $V^c_j=\diag(\tilde{\Lam}^c_j,\overline{\tilde{\Lam}^c_j}),$ $\hat{\Lam}^c_j=\diag(J_{1}(\lam^c_j),\hdots,J_{m_j}(\lam^c_j)),$ $\tilde{\Lam}^c_j=\diag(J_{1}(-\lam^c_j),\hdots,J_{m_j}(-\lam^c_j)),$ $\Lam^c_k=\diag(\hat{\Lam}^c_k,\overline{\hat{\Lam}^c_k}),$ $\hat{\Lam}^c_k=\diag(J_{1}(\lam^c_k),\hdots,J_{m_k}(\lam^c_k)),$ $\Lam^c_r=\diag(\hat{\Lam}^c_r,\tilde{\Lam}^c_r),$ $\hat{\Lam}^c_r=\diag(J_{1}(\lam^c_r),\hdots,J_{m_r}(\lam^c_r)),$ $\tilde{\Lam}^c_r=\diag(J_{1}(-\lam^c_r),\hdots,J_{m_r}(-\lam^c_r)),$ $$\Lam_a=\diag(\Lam_1^a, \hdots, \Lam_{r_1}^a,\, \Lam^a_{r_1+1},\hdots,\Lam^a_{r_2},\,\Lam^a_{r_2+1},\hdots ,\Lam^a_p),$$ 
 $\Lam^a_j=\diag(U1^a_j,\, V1^a_j),\,U1^a_j=\diag(\hat{\Lam}^a_j,\overline{\hat{\Lam}^a_j}),\,V1^a_j=\diag(\tilde{\Lam}^a_j,\overline{\tilde{\Lam}^a_j}),$ $\hat{\Lam}^a_j=\diag(J_{1}(\lam^a_j),\hdots, J_{m_j}(\lam^a_j)),$ $\tilde{\Lam}^a_j=\diag(J_{1}(-\lam^a_j),\hdots,J_{m_j}(-\lam^a_j)),$ $\Lam^a_k=\diag(\hat{\Lam}^a_k,\overline{\hat{\Lam}^a_k}),$ $\hat{\Lam}^a_k=\diag(J_{1}(\lam^a_k),\hdots,J_{m_k}(\lam^a_k)),$ $\Lam^a_r=\diag(\hat{\Lam}^a_r,\tilde{\Lam}^a_r),$ $\hat{\Lam}^a_r=\diag(J_{1}(\lam^a_r),\hdots,J_{m_r}(\lam^a_r)),\,\tilde{\Lam}^a_r=\diag(J_{1}(-\lam^a_r),\hdots,J_{m_r}(-\lam^a_r))$ with $Z=-\ep_1 Z^T\in\R^{n\times n}$ is arbitrary.

Moreover, if $\{\lam^c_j,\,\overline{\lam^c_j},\,-\lam^c_j,\,-\overline{\lam^c_j},\,\lam^c_k,\,\overline{\lam^c_k},\,\lam^c_r,\,-\lam^c_r\,:\,j=1,\hdots, r_1,\,k=r_1+1, \hdots, r_2,\linebreak r=r_2+1, \hdots, p \}\cap \{\lam^f_q\,:\,q=m+1, \hdots, n\}=\emptyset$ 
then a real structured no spillover perturbation of rank $m$ is given by $\Delta A=X_c(\Lam_a-\Lam_c)(X_c^*HX_c)^{-1}X_c^*H $ where $m=\sum_{j=1}^{r_1}4\alpha(\lam^c_j)+\sum_{k=r_1+1}^{r_2}2\alpha(\lam^c_k)+\sum_{r=r_2+1}^{p}2\alpha(\lam^c_r)$. 
\end{theorem}
\begin{proof}
It is given that $AX^c_j=X^c_j\Lam^c_j,\,AX^c_k=X^c_k\Lam^c_k,\,AX^c_r=X^c_r\Lam^c_r$ holds and as $\lam_j^c,\,\overline{\lam_j^c},$\linebreak $-\lam_j^c,\,-\overline{\lam_j^c},\,\lam^c_k,\,\,\overline{\lam_k^c},\,\lam^c_r,\,-\lam^c_r$ are all different so from Proposition \ref{PropA} $(b)$ we have \linebreak $X_c^*HX_c=\diag((X^c_1)^*HX^c_1,\,\dots \,,(X^c_{r_1})^*HX^c_{r_1},\,\,(X^c_{r_1+1})^*HX^c_{r_1+1},\, \dots \,,(X^c_{r_2})^*HX^c_{r_2},$\linebreak $(X^c_{r_2+1})^*HX^c_{r_2+1},\dots,(X^c_p)^*HX^c_p).$ Again as $\left\{\lam^c_j,\,\overline{\lam}^c_j\right\} \cap \left\{-\lam^c_j,\,-\overline{\lam}^c_j\right\} =\emptyset$ so by Proposition \ref{PropA} $(b)$ we have $(Y^c_j)^*HY^c_j=(Z^c_j)^*HZ^c_j=0$ thus $(X^c_j)^*HX^c_j=\bmatrix{0& (Y^c_j)^*HZ^c_j\\\ep_1 ((Y^c_j)^*HZ^c_j)^*& 0}$. Also $\lam^c_j \neq \overline{\lam^c_j}$ therefore from Proposition \ref{PropA} $(b)$ we have $(\hat{X}^c_j)^*H\tilde{X}^c_j= \left(\overline{\hat{X}^c_j}\right)^*H\overline{\tilde{X}^c_j}=0$ hence $(Y^c_j)^*HZ^c_j=\bmatrix{0& (\hat{X}^c_j)^*H\overline{\tilde{X}^c_j}\\ \overline{(\hat{X}^c_j)^*H\overline{\tilde{X}^c_j}}& 0}.$ Hence $(X^c_j)^*HX^c_j=\bmatrix{0& \Gamma_j\\ \ep_1\Gamma_j^*& 0}$ with $\Gamma_j=\bmatrix{0& (\hat{X}^c_j)^*H\overline{\tilde{X}^c_j}\\ \overline{(\hat{X}^c_j)^*H\overline{\tilde{X}^c_j}}& 0}$. 
Further $\lam^c_k \neq \overline{\lam^c_k}$ so by Proposition \ref{PropA} $(b)$ we get $(\hat{X}^c_k)^*H\overline{\hat{X}^c_k}=\left(\overline{\hat{X}^c_k}\right)^*H\hat{X}^c_k=0$, thus $(X^c_k)^*HX^c_k=\mbox{diag}\left( (\hat{X}^c_k)^*H\hat{X}^c_k,\, \overline{(\hat{X}^c_k)^*H\hat{X}^c_k} \right).$
Again as $\lam^c_r\neq -\lam^c_r \in \R$ so from Proposition \ref{PropA} $(b)$ we obtain $(\hat{X}^c_r)^*H\hat{X}^c_r=(\tilde{X}^c_r)^*H\tilde{X}^c_r=0$, hence $(X^c_r)^*HX^c_r=\bmatrix{0& (\hat{X}^c_r)^*H\tilde{X}^c_r\\ \ep_1 ((\hat{X}^c_r)^*H\tilde{X}^c_r)^*& 0}$. It should be noted that as $\lam^c_r\in \R$ so its corresponding Jordan chain consists of real vectors that is the matrices $X^c_{r,l}$ and $\tilde{X}^c_{r,l},$ $l=1,\dots,m_r$ are real. 

Hence on computing we note that $(X_c)^*HX_c(\Lam_a-\Lam_c)=-\ep_1\left(X_c^*HX_c(\Lam_a-\Lam_c)\right)^*$ holds. Thus setting $R$ as the identity matrix the desired condition in equation $(\ref{eqn:cond22})$ is achieved. Therefore by Theorem \ref{updateA} we obtain the structured matrix \begin{eqnarray} \label{updateA_L}
\Delta A = X_c(\Lam_a-\Lam_c)X_c^{\dagger}- H^{-1}(X_c^\dagger)^*(\Lam_a-\Lam_c)^*X_c^*H-H^{-1}(X_c^{\dagger})^*X_c^*HX_c(\Lam_a-\Lam_c)X_c^{\dagger} \nonumber \\  +H^{-1}(I-X_cX_c^{\dagger})Z(I-X_cX_c^{\dagger})\in \L \end{eqnarray}
for which $\{(\lam^a_j,X^c_{j,l})\}_{l=1}^{m_j}$, $\{(\overline{\lam^a_j},\overline{X}^c_{j,l})\}_{l=1}^{m_j}$, $\{(-\lam^a_j,\tilde{X}^c_{j,l})\}_{l=1}^{m_j}$, $\{(-\overline{\lam^a_j},\overline{\tilde{X}}^c_{j,l})\}_{l=1}^{m_j}$, $\{(\lam^a_k,X^c_{k,l})\}_{l=1}^{m_k}$, $\{(\overline{\lam^a_k},\overline{X}^c_{k,l})\}_{l=1}^{m_k}$, $\{(\lam^a_r,X^c_{r,l})\}_{l=1}^{m_r},$ $\{(-\lam^a_r,\tilde{X}^c_{r,l})\}_{l=1}^{m_r}$ are Jordan pairs of $A+\Delta A$ with $Z=-\ep_1 Z^*\in \C^{n \times n}$ is arbitrary. 

Now we prove that $\Delta A$ is a real matrix for $Z^T=-\ep_1 Z \in \R^{n \times n}.$ We note that $\overline{X_c}=X_cR$ and $\overline{X_c(\Lam_a-\Lam_c)}=X_c(\Lam_a-\Lam_c)R$ for $R=\mbox{diag}(R_1,\dots,R_{r_1},\,R_{r_1+1},\dots,R_{r_2},$\linebreak $R_{r_2+1},\dots,R_p )$ with $R_j=\mbox{diag}\left(\tilde{R}_j,\,\tilde{R}_j \right),\,\tilde{R}_j=\bmatrix{0& I_{s_j}\\I_{s_j}& 0},\,R_k=\bmatrix{0& I_{s_k}\\I_{s_k}& 0},\,R_r=I_{2s_r}$ where $s_j=\sum_{l=1}^{m_j} \#(X^c_{j,l}),\,s_k=\sum_{l=1}^{m_k} \#(X^c_{k,l}),\,s_r=\sum_{l=1}^{m_r} \#(X^c_{r,l})$ and $I_u$ denotes the identity matrix of order $u \times u$. Thus $\overline{X_c(\Lam_a-\Lam_c)X_c^{\dagger}}=\overline{X_c(\Lam_a-\Lam_c)}\left(\overline{X_c}\right)^{\dagger}=X_c(\Lam_a-\Lam_c)RR^{-1}X_c^{\dagger}=X_c(\Lam_a-\Lam_c)X_c^{\dagger}$ and $\overline{X_cX_c^{\dagger}}=\overline{X_c}\left(\overline{X_c}\right)^{\dagger}=X_cRR^{-1}X_c^{\dagger}=X_cX_c^{\dagger}$ therefore the matrix $\Delta A$ in equation $(\ref{updateA_L})$ is a real structured matrix for arbitrary $Z=-\ep_1 Z^T\in \R^{n \times n}.$ 

Further, since $\{\lam^c_j,\,\overline{\lam^c_j},\,-\lam^c_j,\,-\overline{\lam^c_j},\,\lam^c_k,\,\overline{\lam^c_k},\,\lam^c_r,\,-\lam^c_r\,:\,j=1,\hdots ,r_1,\,k=r_1+1,\hdots ,r_2,\linebreak r=r_2+1,\hdots ,p \}\cap \{\lam^f_q\,:\,q=m+1,\hdots ,n\}=\emptyset$ then by Theorem \ref{updateA_nospillover} it follows that $X_c^*HX_c$ is nonsingular and we have already verified that $X_c^*HX_c(\Lam_a-\Lam_c)=-\ep_1\left(X_c^* HX_c(\Lam_a-\Lam_c)\right)^*$ holds. Hence the existence of structured solution to problem \textbf{(P1)} is guaranteed by Theorem \ref{updateA_nospillover}. Thus structured no spillover perturbation follows from Theorem \ref{updateA_nospillover}, and it is given by $\Delta A=X_c(\Lam_a-\Lam_c)(X_c^* HX_c)^{-1}X_c^* H\in \L$. Now we will show that $\Delta A$ is real. Note that as $H$ is real so
\beano \overline{\Delta A} &=& \overline{X_c(\Lam_a-\Lam_c)}\left(\left(\overline{X_c}\right)^*H\overline{X_c}\right)^{-1}\left(\overline{X_c}\right)^*H \\ && =X_c(\Lam_a-\Lam_c)R \left(R^*X_c^*HX_cR\right)^{-1}R^*X_c^*H \\&& = X_c(\Lam_a-\Lam_c)RR^{-1}\left(X_c^*HX_c\right)^{-1}(R^*)^{-1}R^*X_c^*H \\ &&=X_c(\Lam_a-\Lam_c)\left(X_c^*HX_c\right)^{-1}X_c^*H=\Delta A,\eeano
hence $\Delta A=X_c(\Lam_a-\Lam_c)\left(X_c^*HX_c\right)^{-1}X_c^*H \in \L$ is a real structured perturbation. This completes the proof. $\hfill{\square}$ 
\end{proof}

The following theorem presents the structured solution to problem \textbf{(P1)} when all the eigenvalues to be changed are simple.

\begin{theorem}
Let $A\in\L\subset\R^{n\times n}.$ Suppose $(\lam^c_j,x^c_j),\,(\overline{\lam^c_j},\overline{x^c_j}),\,(-\lam^c_j,\tilde{x}^c_j),\, (-\overline{\lam^c_j},\overline{\tilde{x}^c_j}),\, (\lam^c_k,x^c_k),\linebreak (\overline{\lam^c_k},\overline{x^c_k})$ and $(\lam^c_r,x^c_r),\,(-\lam^c_r,\tilde{x}^c_r)$ are eigenpairs of $A$ where $j=1, \hdots, r_1,\, k=r_1+1, \hdots, r_2,\,r=r_2+1, \hdots, p$ and $\lam^c_j \neq 0\in \C\smallsetminus (\R\cup i\R),\,\lam^c_k \neq 0\in i\R,\,\lam^c_r \neq 0\in \R.$ 

Let $\lam^a_j, \lam^a_k,\,\lam^a_r$ be a collection of scalars with $\lam^a_j \in \C\smallsetminus (\R\cup i\R),\,\lam^a_k \in i\R$ and $\lam^a_r\in \R$ where  $j=1, \hdots, r_1, \,k=r_1+1,\hdots, r_2,\,r=r_2+1, \hdots,p.$ Let the eigenvalues $\lam_j^c,\,\overline{\lam^c_j},\,-\lam_j^c,\,-\overline{\lam^c_j},\,\lam_k^c,\,\overline{\lam^c_k},\linebreak \lam_r^c,\,-\lam^c_r$ be all simple and distinct. Then any real structured perturbation $\Delta A\in\L$ of $A$ for which $(\lam^a_j,x^c_j),\,(\overline{\lam^a_j},\overline{x^c_j}),\,(-\lam^a_j,\tilde{x}^c_j),\, (-\overline{\lam^a_j},\overline{\tilde{x}^c_j}),\, (\lam^a_k,x^c_k),\,(\overline{\lam^a_k},\overline{x^c_k}),\,(\lam^a_r,x^c_r),\,(-\lam^a_r,\tilde{x}^c_r)$ are eigenpairs of $A+\Delta A$ is given by 
 \beano \Delta A &=& X_c(\Lam_a-\Lam_c)X_c^{\dagger}- H^{-1}(X_c^\dagger)^*(\Lam_a-\Lam_c)^*X_c^*H-H^{-1}(X_c^{\dagger})^*X_c^*HX_c(\Lam_a-\Lam_c)X_c^{\dagger} \\ && +H^{-1}(I-X_cX_c^{\dagger})Z(I-X_cX_c^{\dagger}),\eeano where $X_c=\bmatrix{X_1^c & \hdots & X_{r_1}^c & X_{r_1+1}^c & \hdots & X_{r_2}^c& X^c_{r_2+1}& \hdots & X^c_p},\,X_j^c=\bmatrix{x_j^c & \overline{x_j^c}& \tilde{x}_j^c& \overline{\tilde{x}_j^c}},\linebreak X^c_k=\bmatrix{x_k^c & \overline{x_k^c}},\,X^c_r=\bmatrix{x^c_r& \tilde{x}^c_r},$ $\Lam_c=\diag(\Lam_1^c, \hdots, \Lam_{r_1}^c, \Lam^c_{r_1+1}, \hdots, \Lam_{r_2}^c,\,\Lam^c_{r_2+1},\hdots ,\Lam^c_p),\,\Lam^c_j=\diag(\lam^c_j,\,\overline{\lam^c_j},\,-\lam^c_j,\,-\overline{\lam^c_j}),\,\Lam^c_k=\diag(\lam^c_k,\,\overline{\lam^c_k}),\,\Lam^c_r=\diag(\lam^c_r,\,-\lam^c_r),$ $\Lam_a=\diag(\Lam_1^a, \hdots, \Lam_{r_1}^a,\linebreak \Lam^a_{r_1+1}, \hdots, \Lam_{r_2}^a,\,\Lam^a_{r_2+1},\hdots ,\Lam^a_p),\,\Lam^a_j=\diag(\lam^a_j,\,\overline{\lam^a_j},\,-\lam^a_j,\,-\overline{\lam^a_j}),\,\Lam^a_k=\diag(\lam^a_k,\,\overline{\lam^a_k}),\,\Lam^a_r=\diag(\lam^a_r,\linebreak -\lam^a_r),$ and $Z^T=-\ep_1 Z\in\R^{n\times n}$ is arbitrary.
 
Moreover, if $\{\lam^c_j,\,\overline{\lam^c_j},\,-\lam^c_j,\,-\overline{\lam^c_j},\,\lam^c_k,\,\overline{\lam^c_k},\,\lam^c_r,\,-\lam^c_r\,:\,j=1, \hdots, r_1, k=r_1+1, \hdots, r_2, r=r_2+1, \hdots, p\}\cap \{\lam^f_l\,:\,l=2p+2r_1+1, \hdots, n\}=\emptyset$ 
then a real structured no spillover perturbation of rank $2(p+r_1)$ is given by $\Delta A=X_c(\Lam_a-\Lam_c)(X_c^*HX_c)^{-1}X_c^*H.$ 
\end{theorem} 
\begin{proof}
It is given that $AX^c_j=X^c_j\Lam^c_j,\,AX^c_k=X^c_k\Lam^c_k,\,AX^c_r=X^c_r\Lam^c_r$ holds. Since $\lam_j^c,\,\overline{\lam^c_j},$\linebreak $-\lam_j^c,\,-\overline{\lam^c_j},\, \lam_k^c,\,\overline{\lam^c_k},\,\lam_r^c,\,-\lam^c_r$ are all distinct so using Proposition \ref{PropA} $(b)$ we obtain $X_c^*HX_c=\diag((X^c_1)^*HX^c_1, \hdots,(X^c_{r_1})^*HX^c_{r_1},\,(X^c_{r_1+1})^*HX^c_{r_1+1},\hdots,(X^c_{r_2})^*HX^c_{r_2},\hdots, (X^c_{p})^*HX^c_{p})$. As $\linebreak $ $\lam^c_j\neq 0\in \C\smallsetminus (\R \cup i\R)$ so from Proposition \ref{PropA} $(b)$ we obtain $(X^c_j)^*HX^c_j=\bmatrix{0& \Gamma_j \\ \ep_1 \Gamma_j^*& 0}$ with $\Gamma_j=\bmatrix{0& (x^c_j)^*H\overline{\tilde{x}^c_j}\\ \overline{(x^c_j)^*H\overline{\tilde{x}^c_j}}& 0}.$ Since $\lam^c_k \neq 0\in i\R$ so using Proposition \ref{PropA} $(b)$ we get $(X^c_k)^*HX^c_k=\diag((x^c_k)^*Hx^c_k,\,\overline{(x^c_k)^*Hx^c_k}).$ Again as $\lam^c_r \neq 0$ so from corollary \ref{blockform} we have $(X^c_r)^*HX^c_r=\bmatrix{0& (x^c_r)^*H\tilde{x}^c_r\\ \ep_1 ((x^c_r)^*H\tilde{x}^c_r)^*& 0}.$ It should be noted that $x^c_r$ and $\tilde{x}^c_r$ are real eigenvectors of $A$. Thus it is easy to verify that $X_c^*HX_c(\Lam_a-\Lam_c)=-\ep_1 \left(X_c^*HX_c(\Lam_a-\Lam_c) \right)^*$ holds. Hence the desired result follows from Theorem \ref{P1soln_real_L}. 
$\hfill{\square}$
\end{proof}

In the next theorem we have presented a real structured solution to Problem \textbf{(P1)} for structured matrix $A\in \J \subset \R^{n \times n}$ with real $H=\ep_1 H^T$. 

\begin{theorem} \label{P1soln_real_J}
Let $A\in\J\subset\R^{n\times n}.$ Suppose $\{(\lam^c_i,X^c_{i,l})\}_{l=1}^{m_i}$, $\{(\overline{\lam^c_i},\overline{X}^c_{i,l})\}_{l=1}^{m_i}$, $\{(\lam^c_j,X^c_{j,l})\}_{l=1}^{m_j}$ are Jordan pairs of $A,$ $i=1, \hdots, r,$ $j=2r+1, \hdots, p,$ $\lam^c_i \in \C\smallsetminus \R$ is a nonzero eigenvalue of $A$ and $\lam^c_j\in \R.$ 


Let $\lam^a_i,\, \lam^a_j$ be a collection of scalars with $\lam^a_i\in \C\smallsetminus \R$ and $\lam^a_j \in \R$ where  $i=1, \hdots,r,$ $j=2r+1, \hdots, p.$ Then any real structured perturbation $\Delta A\in\J$ of $A$ for which $\{(\lam^a_i,X^c_{i,l})\}_{l=1}^{m_i},$ $\{(\overline{\lam^a_i}, \overline{X}^c_{i,l})\}_{l=1}^{m_i}$, $\{(\lam^a_j, X^c_{j,l})\}_{l=1}^{m_j}$ are Jordan pairs of $A+\Delta A$ is given by 
 \beano \Delta A &=& X_c(\Lam_a-\Lam_c)X_c^{\dagger}+ H^{-1}(X_c^\dagger)^*(\Lam_a-\Lam_c)^*X_c^*H-H^{-1}(X_c^{\dagger})^*X_c^*HX_c(\Lam_a-\Lam_c)X_c^{\dagger} \\ && +H^{-1}(I-X_cX_c^{\dagger})Z(I-X_cX_c^{\dagger}),\eeano where $$X_c=\bmatrix{X_1^c & \hdots & X_{r}^c & X_{2r+1}^c & \hdots & X_p^c},$$  
 $X_i^c=\bmatrix{\hat{X}_i^c & \overline{\hat{X}_i^c}},$ $\hat{X}_i^c=\bmatrix{X^c_{i,1}\hdots X^c_{i,m_i}},$ $X^c_j=\bmatrix{X^c_{j,1}\hdots X^c_{j,m_j}},$ $$\Lam_c=\diag(\Lam_1^c, \hdots, \Lam^c_{r}, \Lam^c_{2r+1}, \hdots, \Lam_p^c),$$ 
 $\Lam_i^c =\diag\left(\hat{\Lam}_i^c, \overline{\hat{\Lam}_i^c}\right),\,\hat{\Lam}_i^c=\diag(J_{1}(\lam^c_i), \hdots,J_{m_i}(\lam^c_i)),\,\Lam^c_j=\diag(J_{1}(\lam^c_j),\hdots,J_{m_j}(\lam^c_j))$ and $$\Lam_a=\diag(\Lam_1^a, \hdots, \Lam^a_{r}, \Lam^a_{2r+1}, \hdots, \Lam_p^a),$$ $\Lam_i^a =\diag\left(\hat{\Lam}_i^a, \overline{\hat{\Lam}_i^a}\right),$ $\hat{\Lam}_i^a=\diag(J_{1}(\lam^a_i), \hdots,J_{m_i}(\lam^a_i)),\,\Lam^a_j=\diag(J_{1}(\lam^a_j),\hdots,J_{m_j}(\lam^a_j))$ with $Z=\ep_1 Z^T\in\R^{n\times n}$ is arbitrary.
 
Moreover, if $\left\{\lam^c_i,\,\overline{\lam^c_i},\,\lam^c_j\,:\,i=1, \hdots, r,\, j=2r+1, \hdots, p\right\}\cap \{\lam^f_k\,:\,k=q+1, \hdots, n\}=\emptyset$ then a real structured no spillover perturbation is given by $\Delta A=X_c(\Lam_a-\Lam_c)(X_c^*HX_c)^{-1}X_c^*H$ where $q=\sum_{i=1}^{r}2\alpha(\lam^c_i)+\sum_{j=2r+1}^{p}\alpha(\lam^c_j).$ 
\end{theorem}
\begin{proof}
The proof is similar to the proof of Theorem  \ref{P1soln_real_L}. $\hfill{\square}$

\end{proof}

The following theorem presents the structured solution to Problem \textbf{(P1)} when the eigenvalues to be changed are simple.

\begin{theorem} \label{P1_soln_distinctJ}
Let $A\in\J\subset\R^{n\times n}.$ Suppose $(\lam^c_i,x^c_i),\, ( \overline{\lam^c_i},\overline{x}^c_i), (\lam^c_j,x^c_j)$ are eigenpairs of $A$ where $i=1, \hdots, r,$ $j=2r+1, \hdots, p$, and $\lam^c_i\neq 0\in \C\smallsetminus \R,\,\lam^c_j\in \R.$ Let the eigenvalues $\lam_i^c,\, \overline{\lam^c_i},\,\lam_j^c$ be all simple and distinct. Let $\lam^a_i, \lam^a_j$ be a collection of scalars with $\lam^a_i\in \C\smallsetminus \R$ and $\lam^a_j\in \R$ where  $i=1, \hdots, r,$ $j=2r+1, \hdots, p.$ Then any real structured perturbation $\Delta A\in\J$ of $A$ for which  $(\lam^a_i,x^c_i),\, (\overline{\lam^a_i}, \overline{x}^c_i),\, (\lam^a_j, x^c_j)$ are eigenpairs of $A+\Delta A$ is given by 
 \beano \Delta A &=& X_c(\Lam_a-\Lam_c)X_c^{\dagger}+ H^{-1}(X_c^\dagger)^*(\Lam_a-\Lam_c)^*X_c^*H-H^{-1}(X_c^{\dagger})^*X_c^*HX_c(\Lam_a-\Lam_c)X_c^{\dagger} \\ && +H^{-1}(I-X_cX_c^{\dagger})Z(I-X_cX_c^{\dagger}),\eeano where $X_c=\bmatrix{X_1^c & \hdots & X_r^c & x_{2r+1}^c & \hdots & x_p^c},\,X_i^c=\bmatrix{x_i^c & \overline{x}_i^c},$ $\Lam_c=\diag(\Lam_1^c, \hdots, \Lam_r^c,\linebreak \lam^c_{2r+1}, \hdots, \lam_p^c),\,\Lam_i^c =\diag(\lam_i^c, \overline{\lam_i^c}),\,\Lam_a=\diag(\Lam_1^a, \hdots, \Lam_r^a, \lam^a_{2r+1}, \hdots, \lam_p^a),\,\Lam_i^a =\diag(\lam_i^a, \overline{\lam_i^a})$ and $Z^T=\ep_1 Z\in \R^{n\times n}$ is arbitrary.
 
Moreover, if $\left\{\lam^c_i,\,\overline{\lam^c_i},\,\lam^c_j\,:\,i=1,\dots,r,\,j=2r+1,\dots ,p\right\}\cap \{\lam^f_k\,:\,k=p+1,\dots,n\}=\emptyset$ 
then a real structured no spillover perturbation is given by $\Delta A=X_c(\Lam_a-\Lam_c)(X_c^*HX_c)^{-1}X_c^*H.$ 
\end{theorem} 
\begin{proof}
It is given that $AX^c_i=X^c_i\Lam^c_i$ and $Ax^c_j=\lam^c_j x^c_j$ holds. As the eigenvalues $\lam_i^c,\,\overline{\lam_i^c},\, \lam_j^c$ are all distinct so from Proposition \ref{PropA} $(b)$ and using Corollary \ref{blockform} we obtain $X_c^*HX_c=\diag(\Gamma_1,\hdots, \Gamma_r,\,(x^c_{2r+1})^*Hx^c_{2r+1},\hdots,(x^c_{p})^*Hx^c_{p})$ with $\Gamma_i=\bmatrix{0& (x^c_i)^*H\overline{x}^c_i\\ \overline{(x^c_i)^*H\overline{x}^c_i}& 0}$. Hence it is follows that $X_c^*HX_c(\Lam_a-\Lam_c)=\ep_1\left(X_c^*HX_c(\Lam_a-\Lam_c)\right)^*$ holds.  Hence the desired result follows from Theorem \ref{P1soln_real_J}. $\hfill{\square}$
\end{proof}
\section{Numerical Examples}
In this section we consider two numerical examples to validate the obtained results on finding structured perturbations for a structured matrices such that a perturbed matrix reproduce a desired set of eigenvalues and preserve the eigenvectors of the unperturbed matrix. 

\begin{example}

Let $A\in \L\subset \C^{4\times 4}$ corresponding to the scalar product defined by 
\beano H &=& \bmatrix{0& 0& 0& 1\\0& 0& i& 0\\0& -i& 0& 0\\1& 0& 0& 0} \, \mbox{and} \\
 \scriptsize A &=& \bmatrix{1.38328 + 2.23663i&  -1.87526 + 1.09675i&   0.28969 - 1.61767i&   0.00000 - 0.38630i \\
  -0.30572 - 0.81666i&   1.95327 - 0.56098i&   0.70575 + 0.00000i&   1.61767 - 0.28969i\\
   1.70871 + 0.60225i&  -3.36711 + 0.00000i&  -1.95327 - 0.56098i&   1.09675 - 1.87526i \\
   0.00000 - 0.05281i&   0.60225 + 1.70871i&   0.81666 + 0.30572i&  -1.38328 + 2.23663i}.\eeano 
   
Suppose it is required to replace the eigenvalues $2.72646 + 1.45462i,\,-2.72646 + 1.45462i,\,1.39475i$ of $A$ by the desired numbers $3.17634+1.32477i,\,-3.17634+1.32477i,\,-1.30322i$ on perturbing $A$ by $A+\Delta A.$ Then, \beano \Lambda_c &=&\mbox{diag}\left(2.72646 + 1.45462i,\,-2.72646 + 1.45462i,\,1.39475i \right), \\
X_c &=& \bmatrix{
0.73457 + 0.00000i&  -0.02152 + 0.24956i&   0.71237 + 0.00000i\\
-0.27981 - 0.31416i&   0.02238 - 0.04538i&   0.06287 + 0.47116i \\
0.48270 + 0.00893i&   0.81361 + 0.00000i&  -0.05350 - 0.17376i \\
0.20304 - 0.09549i&  -0.51181 + 0.10382i&  -0.46244 - 0.14026i} \, \mbox{and} \\
\Lambda_a &=& \mbox{diag}\left(3.17634+1.32477i,\,-3.17634+1.32477i,\,-1.30322i \right).\eeano

Choose a skew-Hermitian matrix 
\small $$Z=\bmatrix{0.00000 - 1.67851i&   0.13730 + 1.92129i&   1.06091 + 0.54389i&  -1.18529 - 1.28875i \\
  -0.13730 + 1.92129i&   0.00000 + 1.00253i&  -1.43471 + 0.93643i&   0.21830 - 1.51800i\\
  -1.06091 + 0.54389i&   1.43471 + 0.93643i&   0.00000 + 1.05381i&  -0.34779 - 1.04181i \\
   1.18529 - 1.28875i&  -0.21830 - 1.51800i&   0.34779 - 1.04181i&   0.00000 - 0.48090i}.$$

Consequently, by Theorem \ref{P1_comp_distinct}, the corresponding structured perturbation is given by
 $$\Delta A=\bmatrix{-0.13762 - 1.22005i&  -0.65838 + 0.51555i&  -0.12923 + 0.84764i&   0.00000 + 1.62647i\\
   0.72270 - 0.48518i&   0.10142 - 0.64947i&  -0.63261 - 0.00000i&  -0.84764 + 0.12923i\\
   0.02135 + 0.28537i&  -0.72900 + 0.00000i&  -0.10142 - 0.64947i&   0.51555 - 0.65838i\\
   0.00000 + 0.85994i&   0.28537 + 0.02135i&   0.48518 - 0.72270i&   0.13762 - 1.22005i} \in\L.$$
 
Then note that $\|(A+\Delta A)X_c-X_c\Lam_a\|_F=5.4124\times 10^{-15}.$ This ensures that eigenvalues has been replaced successfully without changing its corresponding eigenvectors. 
\end{example}

\begin{example}
Let $A\in \J\subset \R^{5\times 5}$ corresponding to the scalar product defined by 
\beano H &=& \bmatrix{0.90770&   0&   0&   0&  -0.41963\\
   0&   0.99700&   0&   0.07742&   0\\
   0&  0&  -1.00000&   0&   0\\
   0&   0.07742&   0&  -0.99700&   0\\
  -0.41963&   0&  0&  0&  -0.90770}, \, \mbox{and} \\
  A &=& \bmatrix{0.865624&  -1.723920&  -0.349127&   1.693308& 0.330444\\
  -1.766399&   2.575284&   0.927433&   0.347141&  -0.037427\\
  -0.779092&  -0.886132&  -4.893758&  -0.567820&  -2.517258\\
  -1.118777&  -0.275415&  -0.497511&   1.651619&   1.921189\\
  -1.458421&   0.674209&  -2.611835&   1.330580&  -1.574315}.\eeano

Let $\lam^c_1=2.87055 + 0.71763i,\,\lam^c_2=-0.65938$ and $\lam^a_1=3.17331-1.23542i ,\,\lam^a_2=1.33797$ where $\lam^c_1,$ $\overline{\lam^c_1}$ and $\lam^c_2$ are simple eigenvalues of $A.$ Let $(\lam^c_1,x^c_{1}),\,(\overline{\lam^c_1},\overline{x^c_{1}}),\,(\lam^c_2,x^c_{2})$ be eigenpairs of $A$. Suppose it is required to replace the known eigenvalues $\lam^c_1,\,\overline{\lam^c_1},\,\lam^c_2$ of $A$ by the desired numbers $\lam^a_1,\,\overline{\lam^a_1},\,\lam^a_2$ such that $(\lam^a_1,x^c_{1}),\,(\overline{\lam^a_1},\overline{x^c_{1}}),\,(\lam^a_2,x^c_{2})$ are eigenpairs of $A+\Delta A$ when rest of the eigenpairs of $A$ are preserved by $A+\Delta A.$ 

Let \beano \Lambda_c &=& \mbox{diag}\left(2.87055 + 0.71763i,\,2.87055 - 0.71763i,\,-0.65938 \right), \\
 X_c &=& \bmatrix{-0.12653 - 0.25223i&  -0.12653 + 0.25223i&   0.52841 \\
   0.62285 + 0.00000i&   0.62285 - 0.00000i&   0.45037 \\
  -0.20533 + 0.10979i&  -0.20533 - 0.10979i&  -0.46451 \\
   0.47522 - 0.30357i&   0.47522 + 0.30357i&  -0.21212 \\
   0.37734 - 0.13355i&   0.37734 + 0.13355i&   0.50714}, \, \mbox{and} \\
\Lambda_a &=& \mbox{diag}\left(3.17331-1.23542i,\,3.17331+1.23542i,\,1.33797 \right).\eeano

Note that $X_c^*HX_c$ is nonsingular and the assumption of Theorem \ref{P1_soln_distinctJ} holds. Thus a no spillover structured perturbation is given by 
$$\Delta A=\bmatrix{-0.647698&  -1.627577&  -1.550473&  -1.527484& 2.142323\\
  -0.147615&  -4.049645&  -2.545965&   1.524353&   3.829506\\
   0.432213&   2.545140&   1.893505&   0.109324&  -2.759969\\
   1.566449&  -1.994170&  -0.088048&   2.000580&   0.059502\\
  -0.268251&  -3.458913&  -2.323848&   0.444879&   3.406128} \in \J.$$

Further, computing the rest of the `fixed' eigenpairs of $A$ we obtain \beano \Lam_f &=& \mbox{diag}\left(-6.28040,\,-0.17686\right), \, \mbox{and} \\ 
X_f &=& \bmatrix{0.01522&  -0.64955\\
  -0.08140&  -0.55276\\
   0.85330&   0.39494\\
  -0.07067&  -0.01486\\
   0.50993&  -0.34108}.\eeano Then $\|(A+\Delta A)X_f-X_f\Lam_f\|_F=9.59611 \times 10^{-15}$ which ensures that unknown eigenpairs of $A$ remains to be eigenpairs of $A+\Delta A,$ whereas $\|(A+\Delta A)X_c-X_c\Lam_a\|_F=7.31737 \times 10^{-15}$ guarantees that eigenvalues has been replaced successfully without changing its corresponding eigenvectors.  
\end{example}

\noindent{\bf Conclusion.} 
Explicit parametric expressions of structured perturbations of a structured matrix are obtained such that several given eigenvalues of the matrix can be modified by a desired set of eigenvalues while preserving the Jordan chains of the unperturbed matrix. Structured preserving perturbations  are determined which in addition preserve the rest of the Jordan pairs of the matrix that need not be known. In this case the perturbations are called no spillover perturbations. These results are obtained by first determining structure preserving perturbations which preserve complementary invariant subspaces of a given structured matrix.

\end{document}